\documentclass[12pt,A4]{article}
\usepackage[applemac]{inputenc}
\usepackage[english]{babel}
\usepackage{amsmath,amsfonts,amssymb,amsthm,mathrsfs,bbm}
\usepackage[font=sf, labelfont={sf,bf}, margin=1cm]{caption}
\usepackage{graphicx,graphics}
\usepackage{epsfig}
\usepackage{latexsym}
 \usepackage{ae,aecompl}
\usepackage{pstricks}
\usepackage{enumerate}
\usepackage{xcolor}
\usepackage{pifont}

\usepackage[pdfpagemode=UseNone,bookmarksopen=false,colorlinks=true,urlcolor=blue,citecolor=blue,citebordercolor=blue,linkcolor=blue]{hyperref}
\usepackage[normalem]{ulem}
\usepackage[top=2.5cm,left=1.5cm,right=1.5cm,bottom=2.5cm]{geometry}
\usepackage{comment}
\usepackage{eulervm,palatino}

\DeclareGraphicsExtensions{.jpg,.pdf}

\newcommand{\op}[1]{\operatorname{#1 }}

\newcommand{\R}{\mathbb R}
\newcommand{\N}{\mathbb N}

\renewcommand{\P}[2]{\mathbb{P}_{#2}\left( #1 \right)}

\def\build#1_#2^#3{\mathrel{
\mathop{\kern 0pt#1}\limits_{#2}^{#3}}}

\newtheorem{theorem}{Theorem}

\newtheorem{proposition}[theorem]{Proposition}
\newtheorem{lemma}[theorem]{Lemma}
\newtheorem{corollary}[theorem]{Corollary}

\newtheorem*{conjecture}{Conjecture}
\newtheorem*{open}{Open Question}

\renewcommand{\L}{\mathbb{L}}

\renewcommand{\P}[1]{\mathbb{P}\left[#1\right]}
\theoremstyle{definition}
\newtheorem*{remark}{Remark}

\newcommand\Es[1]{\mathbb{E}\left[#1\right]}

\def \N {\mathbb N}
\def \R {\mathbb R}

\def \P {\mathbb{P}}

\def \dTV  {\mathrm{d_{TV}}}
\def \dGH  {\mathrm{\mathrm{d_{GH}}}}

\def \rab {\mathsf{Glu}}

\newcommand{\pol}{\mathrm{P\'ol}}
\def \bt {\boldsymbol { \tau}}
\def \btp {\boldsymbol { \tau'}}
\def \bphi {\boldsymbol { \phi}}
\def \bpsi {\boldsymbol { \psi}}

\usepackage{enumerate}
\def\llbracket{[\hspace{-.10em} [ }
\def\rrbracket{ ] \hspace{-.10em}]}
\title{  \vspace {-2cm}\textbf{Scaling limits and influence of the seed graph in preferential attachment trees}}
\date{}
\author{Nicolas Curien\thanks{CNRS and LPMA, Université Pierre et Marie Curie (Paris 6). \hfill  \texttt{nicolas.curien@gmail.com}} 
\quad Thomas Duquesne\thanks{LPMA, Université Pierre et Marie Curie (Paris 6).\hfill \texttt{thomas.duquesne@upmc.fr}} 
\quad Igor Kortchemski\thanks{DMA, École Normale Supérieure.\hfill  \texttt{igor.kortchemski@normalesup.org}} 
\quad and Ioan Manolescu\thanks{Département de Mathématiques, Université de Genève.\hfill  \texttt{ioan.manolescu@unige.ch}
}}

\DeclareSymbolFont{extraup}{U}{zavm}{m}{n}
\DeclareMathSymbol{\varheart}{\mathalpha}{extraup}{86}
\DeclareMathSymbol{\vardiamond}{\mathalpha}{extraup}{87}

\makeatletter
\renewcommand*{\@fnsymbol}[1]{\ensuremath{\ifcase#1\or  \spadesuit \or \varheart\or \vardiamond \or \clubsuit \or
   \mathsection\or \mathparagraph\or \|\or **\or \dagger\dagger
   \or \ddagger\ddagger \else\@ctrerr\fi}}
\makeatother

\newcommand{\PP}{\mathbb P}

\begin{document}

\maketitle

\let\thefootnote\relax\footnotetext{ \\\emph{MSC2010 subject classiﬁcations}. Primary 05C80, 60J80; secondary 05C05, 60G42. }
 
\vspace {-0.5cm}

\begin{abstract} 

\medskip

We are interested in the asymptotics of random trees built by linear preferential attachment, 
also known in the literature as Barab\'asi--Albert trees or plane-oriented recursive trees. 
We first prove a conjecture of Bubeck, Mossel \& R\'acz \cite{BMR14v3} 
concerning the influence of the seed graph on the asymptotic behavior of such trees. 
Separately we study the geometric structure of nodes of large degrees 
in a plane version of Barab\'asi--Albert trees \emph{via} their associated looptrees. 
As the number of nodes grows, we show that these looptrees, appropriately rescaled, 
converge in the Gromov--Hausdorff sense towards a random compact metric space which we call the Brownian looptree. 
The latter is constructed as a quotient space of Aldous' Brownian Continuum Random Tree 
and is shown to have almost sure Hausdorff dimension $2$. 
\end{abstract}

 \begin{figure}[!h]
 \begin{center}
  \includegraphics[width=0.45 \linewidth]{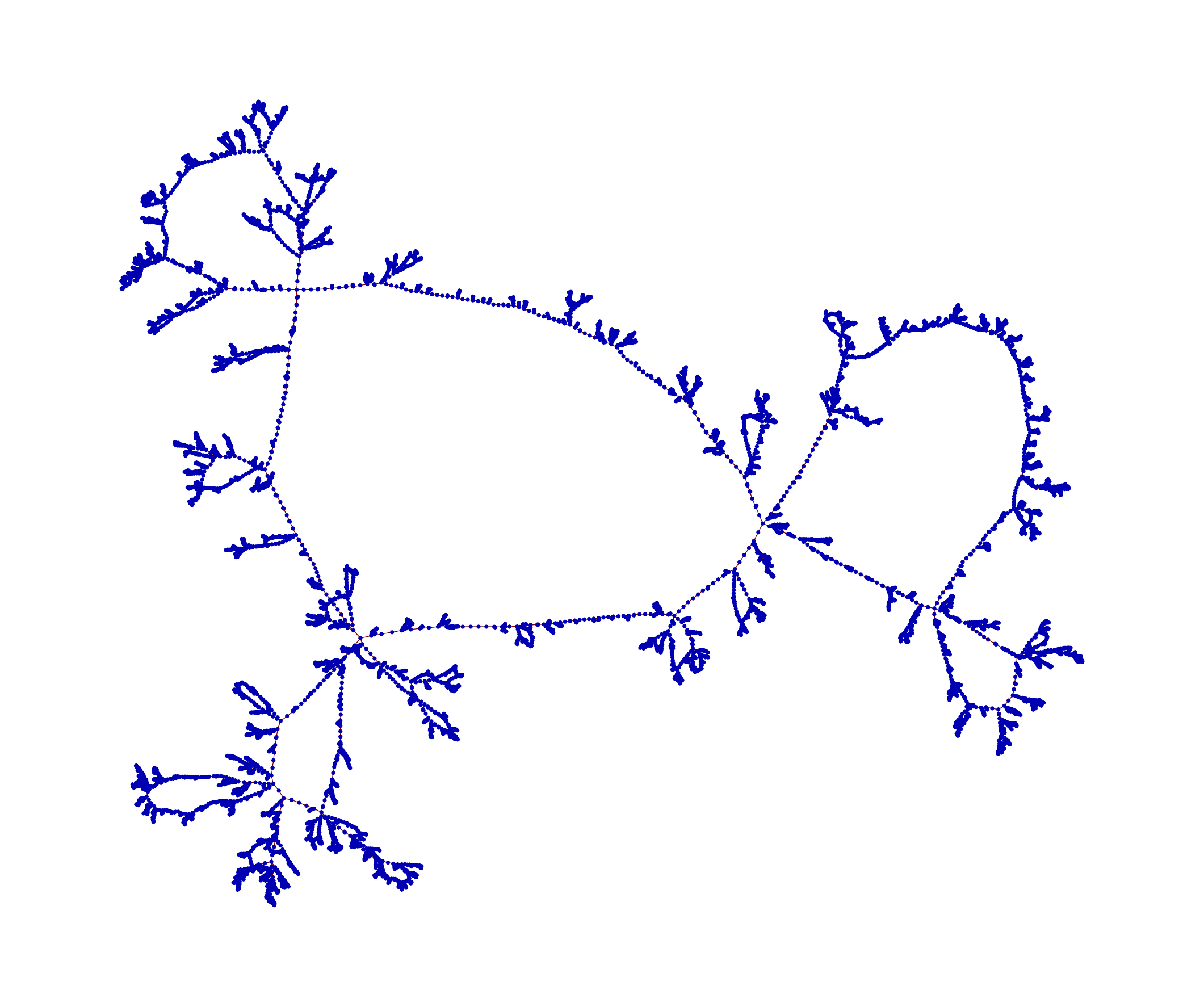}
 \caption{ \label{fig:largediss} The looptree associated with a large plane Barab\'asi--Albert tree.}
 \end{center}
 \end{figure}
 
 \vfill \pagebreak
\section{Introduction}

Random graphs constructed recursively by preferential attachment rules have attracted a lot of attention in the last decade. They are sensible models for many real-world networks, and have the remarkable scale-free property, meaning that their degree distribution exhibits a power law behavior. The literature on the subject is extremely vast, and we refer to \cite{Rem13} for an overview and references. 

In this work, we focus on the simplest and the best known of these models, the linear preferential attachment model (LPAM in short). Starting with a finite tree $T_{1}$ (i.e.~a finite connected graph without cycles, considered up to graph isomorphisms), one constructs recursively a sequence of random trees $T_{1},T_{2}, \ldots $ by requiring that for $i \geq 1$, the tree $T_{i+1}$ is obtained from the tree $T_{i}$ by joining with an edge a new vertex with a random vertex of $T_{i}$, chosen proportionally to its degree. These trees are also known in the literature as plane-oriented recursive trees. This model was introduced by Szym\'anski \cite{Szy87}, and generalized and popularized by Albert \& Barab\'asi \cite{AB99} and Bollob\'as, Riordan, Spencer \& Tusn\'ady \cite{BRST01}. 

This work concerns two related aspects of the LPAM. 
First we investigate the influence of the initial tree 
(also called the seed) on the behavior of $T_{n}$ as $n \rightarrow \infty$.
Next we study the graph structure of $T_{n}$ as $n \rightarrow \infty$ by studying its associated looptree (see below for the definition of a looptree associated with a tree).

\paragraph{Influence of the seed graph.}  Bubeck, Mossel and R\'acz \cite{BMR14v3} recently raised the question of the influence of the initial tree 
on the large time behavior of the LPAM. More precisely, given a tree $S$ with $|S| = n_0 \geq 2$ vertices, consider the sequence of trees $  ({T}^{(S)}_{n})_{n \geq n_0}$ constructed by using the previously mentioned preferential attachment rule and starting with $ {T}_{n_0} =S$. The tree $S$ is called the \emph{seed graph}. Informally, the question is whether the seed graph can be determined from the law of $T_n$ for large values of $n$. Following \cite{BMR14v3}, for finite trees $S_{1}$ and $S_{2}$, set
$$ d(S_{1},S_{2}) = \lim_{n\to \infty} \mathrm{d_{TV}}(  {T}_{n}^{(S_{1})},  {T}_{n}^{(S_{2})}),$$
where $\dTV$ denotes the total variation distance for random variables taking values in the space of finite trees. 
Bubeck, Mossel and R\'acz  \cite{BMR14v3} have observed that $d$ is a pseudo-metric and have conjectured that $ d$ is a metric in non-trivial cases. We confirm this conjecture:
 
\begin{theorem}\label{thm:metric}
The function $d$ is a metric on 
trees with at least $3$ vertices.
\end{theorem}

Observe that Theorem~\ref{thm:metric} means that  $ \mathrm{d_{TV}}({T}_{n}^{(S_{1})},  {T}_{n}^{(S_{2})})$ remains bounded away from $0$ as $n \rightarrow \infty$, as soon as the two seeds $S_{1},S_{2}$ are different and consist of at least $3$ vertices. In  \cite{BMR14v3} this is proved for seeds with different degree sequences by studying the asymptotic behavior of the tail of the degrees of the vertices of $T_{n}^{(S)}$,  and the authors notice that additional information concerning the graph structure has to be incorporated to solve the general case. To this end, they suggest to study the maximum of the sum of the degrees over all embeddings of a fixed tree in $T_{n}^{(S)}$.

{In order to establish Theorem~\ref{thm:metric}, we design another a family of "observables" of $T^{(S)}_{n}$, indexed by finite trees $\tau$ which roughly correspond to the {total} number of possible embeddings of a given tree $\tau$ into $T_{n}^{(S)}$.  {Using these variables, we then construct a family of martingales such that} their laws differ asymptotically for different seed graphs.} 

{Although rather implicit in our proof of Theorem~\ref{thm:metric}, 
the {underlying} key feature of the LPMA is the geometric structure induced by the nodes of large degree in $T_{n}^{(S)}$.  
It is known that the maximal degree in $T_{n}^{(S)}$ is of order $ \sqrt{n}$ (see e.g. \cite{Mor05}) 
and that there is a tight number of vertices with degree of this order.  
Roughly speaking, the geometric tree structure induced by these vertices
is captured by the martingales constructed for the proof of Theorem~\ref{thm:metric}.
In this spirit, our second main result is devoted to giving a  {precise} sense to the continuous
{scaling limit of} this structure through the looptree associated with $T_{n}^{(S)}$. 
As we will see below, the looptree of a plane tree encodes in a natural way the geometric structure of nodes of large degree.}

 \paragraph{Scaling limits of looptrees.} For our next results, we consider the planar version of the LPAM.
 For a plane (i.e.~embedded in the plane) tree $S$ with $n_0$ vertices, 
 consider the sequence of random plane trees $(\overline{T}_n^{(S)})_{n \geq n_{0}}$
 defined by $\overline{T}_{n_0}^{(S)} = S$ and, for $n \geq n_{0}$, conditionally on $\overline{T}_{n_0}^{(S)}, \dots, \overline{T}_{n}^{(S)}$,
 $\overline{T}_{n+1}^{(S)}$ is obtained by grafting an edge leading to a new vertex 
 inside a uniformly chosen corner of $\overline{T}_{n}^{(S)}$
 (by definition, a corner is an angular sector in the plane formed by two consecutive half-edges around a vertex).
 Since the number of corners around a vertex is equal to its degree, 
 it is immediate that the tree structure of $(\overline{T}_n^{(S)})_{n \geq n_{0}}$ is that of a LPAM. Thus we no longer distinguish between $\overline{T}_n^{(S)}$ and $T_n^{(S)}$.

The plane embedding of the LPAM allows us to consider its associated looptree. 
The notion of looptree was introduced in \cite{CK13} 
(see also \cite{CK13+} for the appearance of looptrees in the context of random maps). 
Informally speaking, the looptree $ \mathsf{Loop}(\tau)$ of a plane tree $\tau$ 
is the graph constructed by replacing each vertex $u \in \tau$ 
by a discrete cycle of length given by the degree of $u$ in $ \tau$ 
and gluing these cycles according to the tree structure of $\tau$, see Fig.~\ref{fig:loop}. 
See \cite{CK13} for a formal definition. 
One may view $ \mathsf{Loop}( \tau)$ as a compact metric space by endowing the set of its vertices with the graph distance. 

\begin{figure}[!h]
 \begin{center}
 \includegraphics[width=10cm]{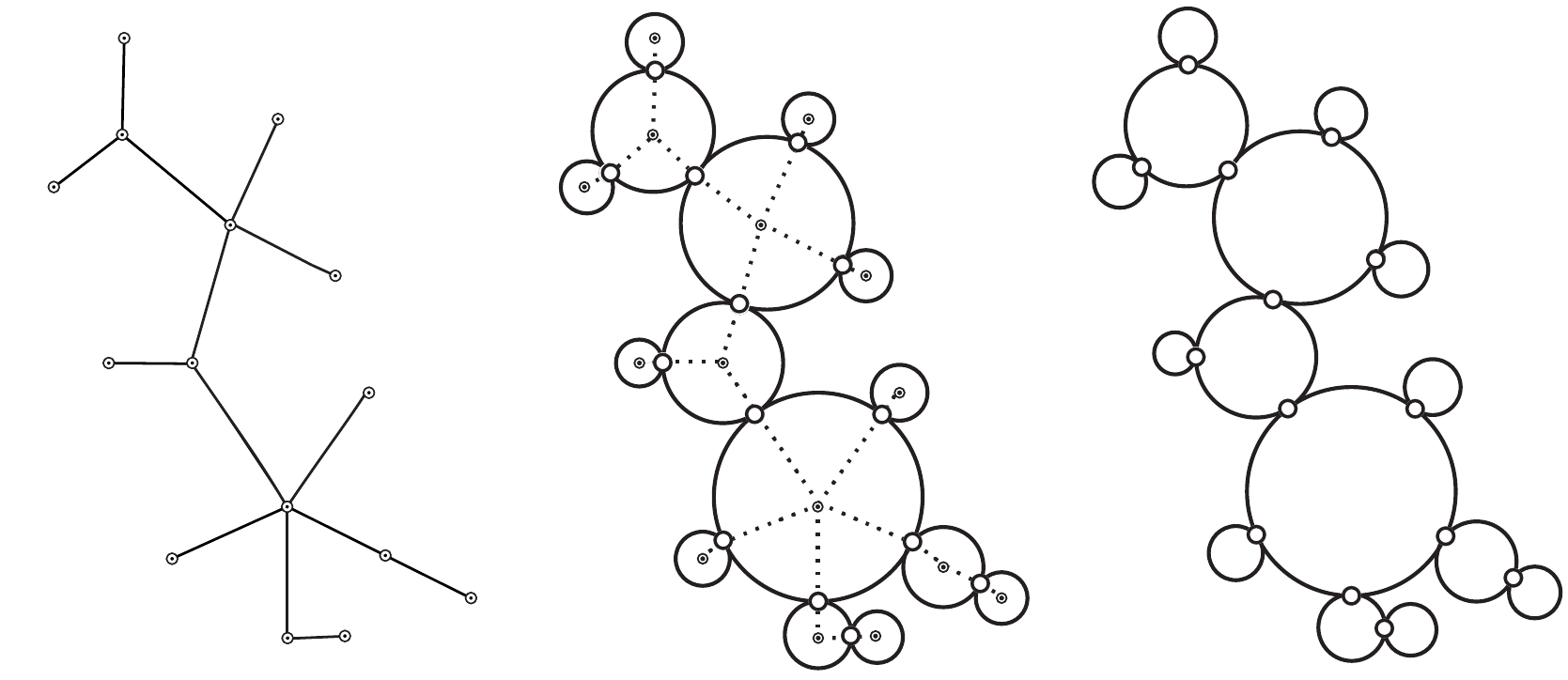}
 \caption{An example of the looptree associated with a plane tree.}
 \label {fig:loop}
 \end{center}
 \end{figure}

\newcommand{\thefootnote}{\arabic{footnote}}	

We will show that, for a plane tree $S$, the sequence of compact metric spaces $(\mathsf{Loop}({T}_{n}^{(S)}))_{n \geq |S|}$, 
suitably rescaled by a factor $ n^{-1/2}$, converges towards a random compact metric space. 
The latter convergence is almost sure with respect to the Gromov--Hausdorff 
topology of compact metric spaces; see Section~\ref{sec:GH} for background. 

It will be useful to consider for a start the case of the particular seed graph $\multimap$ 
consisting of a single vertex with a unique corner.
Formally, $ \multimap$ is a \emph{planted} tree. 
By definition, a tree $ \tau$ is planted if a distinguished half-edge is attached to a vertex of $ \tau$ 
(thus increasing the degree of this vertex by one and adding a corner to it). 
One defines the sequence of random planted plane trees $( {T}_{n}^{\multimap})_{n \geq 1}$ 
by the preferential attachment rule described above, starting with the seed graph $\multimap$ 
(to simplify notation we write ${T}_{n}^{\multimap}$ instead of ${T}_{n}^{(\multimap)}$). 
See Fig.~\ref{fig:ABplanaire} for an illustration. 

\begin{figure}[!h]
 \begin{center}
 \includegraphics[width=0.9 \linewidth]{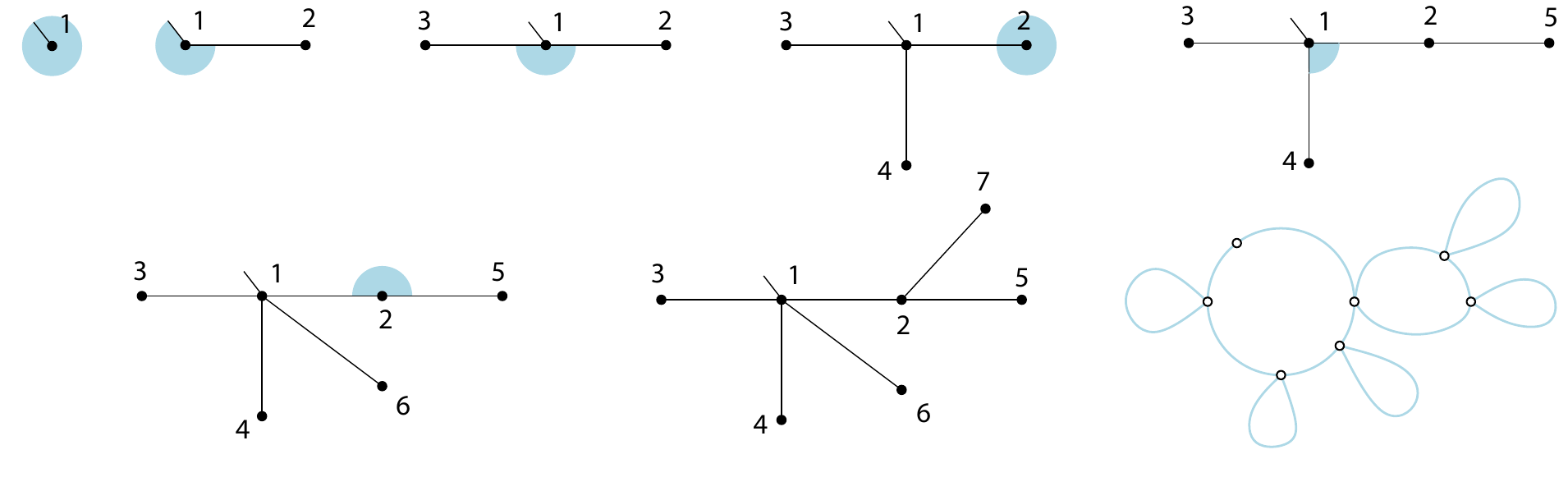}
 \caption{Illustration of the growth mechanism of the plane LPAM started from $\multimap$. 
 At each step, the corner in which a new edge is grafted is highlighted in light blue. 
 The last diagram is the looptree of the last planted tree displayed.}
\label{fig:ABplanaire}
 \end{center}
 \end{figure}
 
The looptree associated with a planted tree $ \tau$ is the looptree of the tree obtained by adding a new vertex to the endpoint of the half-edge of $ \tau$, but where the self-loop surrounding this new vertex is removed, see Fig.~\ref{fig:ABplanaire} for an example. 
If $ \mathcal{M}$ is a metric space, we write $ c \cdot  \mathcal{M}$ for the metric space obtained from $ \mathcal{M}$ by multiplying all distances by $c>0$.
Our result second main result is the following. 

\begin{theorem} \label{thm:scaling} 
	The following convergence holds almost surely in the Gromov--Hausdorff topology
  	\begin{eqnarray*} 
		n^{-1/2} \cdot \mathsf{Loop}( T_{n}^\multimap) & \xrightarrow[n\to\infty]{a.s.} & 2 \sqrt{2} \cdot  \mathcal{L},  
	\end{eqnarray*} 
	where $\mathcal{L}$ is a random compact metric space called the Brownian looptree. 
\end{theorem}

\begin{remark} 
It is natural to scale $\mathsf{Loop}(T_{n}^\multimap)$ by a factor $ n^{-1/2}$ 
in order to obtain a non-degenerate limiting compact metric space.
Indeed, lengths of loops in $\mathsf{Loop}( T_{n}^\multimap)$ correspond to vertex degrees of $T^\multimap_{n}$, 
and it is well-known that the  maximum degree of $T_{n}^\multimap$ is of order $ \sqrt{n}$ (see e.g. \cite{Mor05}). 
Moreover the diameter of $T_n^{\multimap}$ is of much lower order, namely of order $ \log(n)$ (see e.g.\,\cite[Sec.~11]{Rem13}).
In light of the above, it is not surprising that looptrees associated with $T_{n}^\multimap$ admit a nontrivial scaling limit, while the trees $T_{n}^\multimap$ themselves do not. 
\end{remark}

The metric space $ \mathcal{L}$ is constructed as a quotient of the Brownian Continuum Random Tree (in short the CRT) which was introduced by Aldous in \cite{Ald91a}. Let us give for the moment a heuristic construction of $ \mathcal{L}$. Denote by $ \mathcal{T}_{ \mathbf{e}}$ the CRT obtained from a Brownian excursion $ \mathbf{e}$ (see e.g. \cite[Sec.~2]{LG05}). This random tree supports a natural mass measure $\mu$. This is a probability measure on $ \mathcal{T}_{ \mathbf{e}}$ and is supported by the leaves of $ \mathcal{T}_{ \mathbf{e}}$.
Denote by $(X_{i})_{i \geq 0}$ a sequence of i.i.d.~points sampled according to $\mu$. 
For every $n \geq 2$, consider the subtree $ \mathsf{Span}( \mathcal{T}_{ \mathbf{e}} ; X_{0}, \ldots , X_{n})$ of $ \mathcal{T}_{ \mathbf{e}}$
spanned by $X_{0}, X_{1},\ldots , X_{n}$ (see Sec.~\ref {sec:br} for a precise definition). Denote by $P_{n}$ the point in $ \mathsf{Span}( \mathcal{T}_{ \mathbf{e}} ; X_{0}, \ldots , X_{n-1})$ which is the closest to $X_{n}$, see Fig.~\ref{fig:rabbit2}. 
Set also $P_{1}=X_{0}$. 
Informally, the compact metric space $ \mathcal{L}$ 
is obtained from $ \mathcal{T}_{ \mathbf{e}}$ by making the point identifications $X_{n} \sim P_{n}$ for every $n \geq 1$. 
See Section~\ref{section:scaling} for the rigorous construction.

\begin{figure}[!h]
 \begin{center}
 \includegraphics[width=0.9 \linewidth]{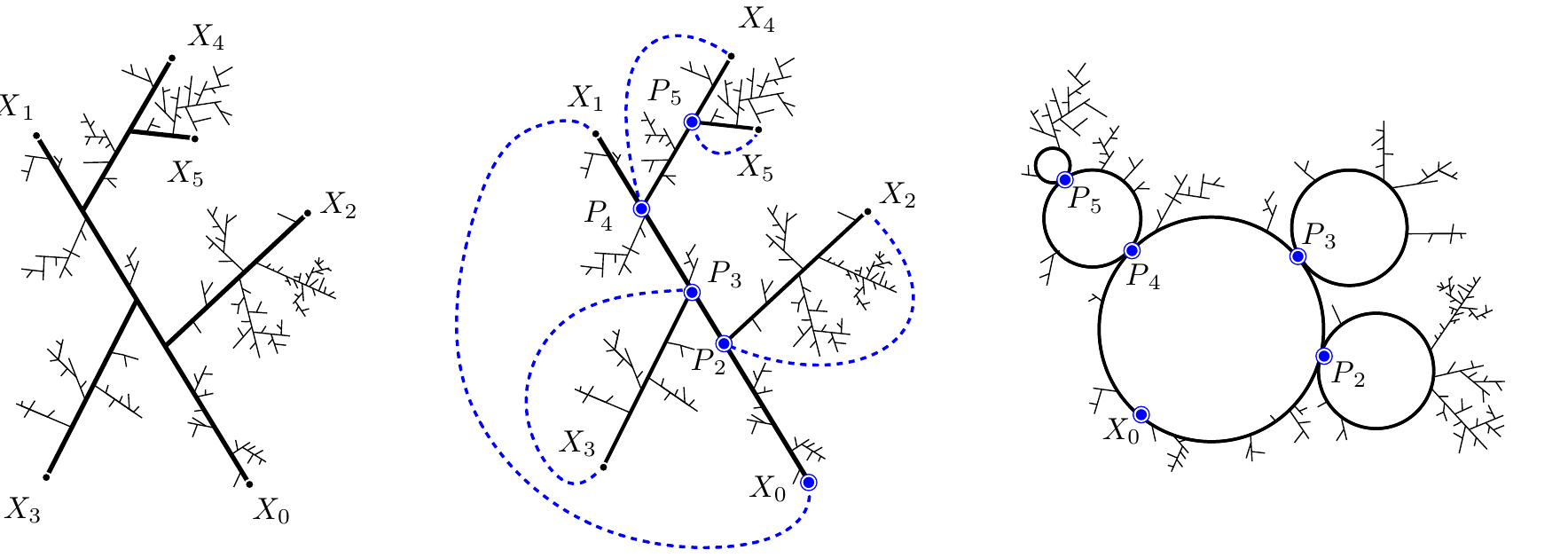}
 \caption{The metric space obtained from $ \mathcal{T}_{ \mathbf{e}}$ 
 by making the identifications $X_{n} \sim P_{n}$ for $1 \leq n \leq 5$.}
 \label{fig:rabbit2}
 \end{center}
\end{figure}

Since the sizes of loops in $ \mathsf{Loop}({T}_{n}^\multimap)$ correspond to vertex degrees in $ {T}_{n}^\multimap$, 
$\mathcal{L}$ contains the limiting joint distribution of the scaled degrees in $ T_{n}^\multimap$.
This distribution has been computed in \cite{PRR14} and asymptotic estimates on its tails studied in \cite {BMR14v3}.
But $\mathcal{L}$ incorporates additional information concerning the graph structure of $ T_{n}^\multimap$.

An important tool is a coupling between the LPAM and Rémy's algorithm \cite{Rem85} which appeared in \cite{PRR14}.  The proof of Theorem~\ref{thm:scaling} combines this coupling with the convergence of scaled uniform binary trees towards the Brownian CRT.

For planar LPAM's starting with a generic seed graph $S$, 
we obtain as consequence of Theorem~\ref{thm:metric} a convergence similar to that for the seed $\multimap$. 

\begin{corollary}\label{cor:scaling} 
	For any plane tree $S$ there exists a random compact metric space $\mathcal{L}^{(S)}$ such that following convergence 
	holds almost surely for the Gromov--Hausdorff topology
	\begin{eqnarray*}
		n^{-1/2} \cdot \mathsf{Loop}(  T_{n}^{(S)}) & \xrightarrow[n\to\infty]{a.s.} &   2 \sqrt{2} \cdot  \mathcal{L}^{(S)}. 
	\end{eqnarray*} 
\end{corollary}  

The limiting metric space $\mathcal{L}^{(S)}$
is constructed by gluing weighted i.i.d.~copies of $\mathcal{L}$; see Section~\ref{sec:genseed} for details.  
In light of Theorem~\ref{thm:metric}, 
we expect that if $S_1 \neq S_2$ are different seed graphs with at least three vertices, 
then the laws of $\mathcal{L}^{(S_{1})}$ and of $\mathcal{L}^{(S_{2})}$ are different 
and we further conjecture (see Section~\ref{sec:genseed})
that the distance $d(\cdot,\cdot)$ appearing in Theorem~\ref{thm:metric} can be expressed as 
 \begin{eqnarray} \label{eq:conjintro} d( S_{1},S_{2}) = \mathrm{d_{TV}}( \mathcal{L}^{(S_{1})}, \mathcal{L}^{(S_{2})}).  \end{eqnarray}

 \medskip 
 Several other random compact metric spaces have been constructed as quotients of the Brownian CRT 
and appear as limits of discrete structures. 
For example, the scaling limit of the connected components of the Erd\H{o}s-R\'enyi random graph 
is described by a tilted Brownian CRT with a  finite number of point identifications \cite{ABBG12}. 
These conserve many of the properties of the CRT (such as the Hausdorff dimension equal to $2$). 
Another example is the Brownian map, obtained from the CRT by gluing a continuum number of points 
using additional randomness involving Brownian motion indexed by the CRT, see \cite{LGICM}. 
In this case, the structure of the metric space is drastically altered by the identifications,
and it is known that the Brownian map is almost surely homeomorphic to the sphere 
and has Hausdorff dimension $4$ (see \cite{LGICM}). 
The Brownian looptree is, in some sense, in-between the two examples above, 
since it involves a countable number of point identifications in the CRT, 
which change completely its topological structure, 
but conserve the Hausdorff dimension of the CRT. 
 
\begin{proposition}\label{prop:hausdorff} 
	Almost surely, the Hausdorff dimension of $ \mathcal{L}$ is $2$.
\end{proposition}
 
We mention that in \cite{CK13}, a related one-parameter family of random compact metric spaces  $(\mathscr{L}_{ \alpha})_{ \alpha \in (1,2)}$ has been constructed. They are called stable looptrees, and  appear as scaling limits of discrete looptrees associated with large critical Galton--Watson trees whose offspring distribution belongs to the domain of attraction of an $\alpha$-stable law. The Brownian looptree introduced in this work differs substantially from stable looptrees. For example, in the Brownian looptree, large loops are adjacent, while in $\mathscr{L}_{ \alpha}$ large loops are connected through infinitely many microscopic loops. In addition, in \cite{CK13} it is shown that the Hausdorff dimension of $ \mathscr{L}_{ \alpha}$ is almost surely $ \alpha<2$. 

\bigskip 
We believe that, as illustrated by Theorem~\ref{thm:scaling}, looptrees are an interesting means to give a sense to scaling limits of highly dense random trees. See in particular, Section~\ref{sec:glpm} for a conjecture concerning affine preferential attachment models and random trees built by Ford's algorithm.  We hope to pursue this line of research in future work.

\paragraph*{Outline.} 
The paper is organized as follows.  
Sections~\ref{sec:influence} and ~\ref{section:scaling} establish Theorems~\ref{thm:metric} and~\ref{thm:scaling}, respectively. 
In Section~\ref{sec:influence}, we first define the observables that we use, then prove Theorem~\ref{thm:metric}. 
In Section~\ref{section:scaling} we start by presenting the connection between the plane LPAM and R\'emy's algorithm,  
then construct the Brownian looptree from the Brownian CRT and prove Theorem~\ref{thm:scaling} and its corollary. 
We end the section with the computation of the Hausdorff dimension of the Brownian looptree. 
These two sections are largely independent. 
Finally, in Section~\ref{sec:ext}, we propose several extensions and generalizations.

\section{Influence of the seed graph} \label{sec:influence}

In this section, we assume that the LPAM is started from a seed graph $S$,
which is a (non planted) tree with at least two vertices. 
In particular, the total degree of $T_{n}^{(S)}$ (that is  the sum of the degrees of all its vertices) is always equal to $2n-2$. 

\newcommand{\EE}{\mathbb{E}}
\newcommand{\ind}{\mathbbm{1}}

\subsection{Decorated trees}\label{sec:decorated}

\newcommand{\bs}{\boldsymbol { \sigma}} 
\newcommand{\br}{\boldsymbol { \rho}}

A decorated tree is a pair $\boldsymbol { \tau}=( \tau, \ell)$
consisting of a tree $ \tau$ and a family of positive integers $(\ell(u); u \in \tau)$ carried by its vertices.
We denote by $|\boldsymbol { \tau}|$ the total number of vertices of $\tau$ 
and set $w(\boldsymbol { \tau}) := \sum_{u \in \tau}\ell(u)$ to be the total weight of $ \bt$. 
We insist on the fact that $ \ell(u)>0$ for every $u \in \tau$.

Let $ \mathcal{D}$ be the set of all decorated trees. 
For $ \bt, \btp \in \mathcal{D}$,  we write $  \bt \prec \btp$ 
if $w(\bt) < w(\btp)$ and $|\bt | \leq |\btp|$  or if $w(\bt) = w(\btp)$ and  $|\bt | < |\btp|$.
Thus $\prec$ is a strict partial order on $\mathcal{D}$ {and we denote by $\preccurlyeq$ the associated partial order}. 

We now define the observables which will be {used} to identify the seed of a LPAM.  
For $k,j\geq 1$, write $[k]_{j}= k(k-1) \cdots ( k-j+1)$. If $ \tau,T$ are trees, we say that a map $ \phi: \tau \rightarrow T$ is an embedding if $ \phi$ is an injective graph homomorphism.
For a decorated tree $\bt$, set
$$D_{\bt}(T) = \sum_{ \phi} \prod_{u \in \tau} [\deg_{T}\phi(u)]_{ \ell(u)},$$ 
where the sum is taken over all embeddings $ \phi : \tau \rightarrow T$ and where $ \deg_{T}(x)$ denotes the degree of a vertex $x \in T$. 
When $\tau = \text{\ding{172}}$ is the decorated tree formed of a single vertex with label one, $D_{\text{\ding{172}}}(T)$ is just the total degree of~$T$. Theorem~\ref{thm:metric} is a consequence of the following proposition.

\begin{proposition} \label{prop:martingales} 
	Let $\bt$ be a decorated tree.
	There exist constants $\{c_n(\bt, \btp):  \btp \preccurlyeq \bt, n \geq 2\}$ with $c_{n}(\bt,\bt)>0$ such that, 
	for every seed $S$, the process $(M^{(S)}_{\bt}(n))_{n \geq n_0}$ defined by
	$$ M^{(S)}_{\bt}(n) = 
	 \sum_{\btp \preccurlyeq \bt} c_n(\bt, \btp) \cdot D_{\btp}(T_{n}^{(S)}) $$
	is a martingale with respect to the filtration $ \mathcal{F}_{n}= \sigma( T_{n_{0}}^{(S)}, \ldots,   T_{n}^{(S)})$
	and is bounded in $\L^2$.
\end{proposition}

\begin{remark}
Rather than the quantities $D_{\bt}$ defined above, 
a more natural family of observables to consider are the number $E_\tau$ of embeddings of a tree $\tau$ inside $T$.   
These observables could indeed be used to distinguish between seeds of the LPAM (the martingales $M$ of Proposition~\ref{prop:martingales} could be written in terms of $E_\tau$ only).  However, as we will see, the main advantage of the observables $D_{\bt}$ is that they are more amenable for recurrence relations (see Lemma~\ref{lem:espcond}).
\end{remark}

The quantity $D_{\bt}(T)$ has a special interpretation for plane trees $T$. 
Imagine that there are $\ell(u)$ distinguishable arrows pointing to each vertex $u \in \tau$.
Then $D_{\bt}(T)$ is the number of ways to embed $\tau$ in $T$ 
in such a way that each arrow pointing to a vertex of $\tau$ 
is associated with a corner of $T$ adjacent to the corresponding vertex,
with distinct arrows associated with distinct corners. 
We call this type of embeddings decorated embeddings. 

Proposition~\ref {prop:martingales} is the main ingredient in the proof of Theorem~\ref {thm:metric};
its proof occupies the following subsections. 
Before, let us explain how to deduce Theorem~\ref{thm:metric} from Proposition~\ref {prop:martingales}.

\begin {proof}[Proof of Theorem~\ref {thm:metric}] 
	Let $S_{1} \neq S_{2}$ be two distinct trees with at least $3$ vertices. We claim that if $n_{0} = \max( |S_{1}|, |S_{2}|)$, then 
	there exists a deterministic decorated tree $\bt$ such that 
	\begin{equation}\label{eq:bt}
		\Es{ D_{\bt}(T_{n_{0}}^{(S_{1})} )} \quad \neq  \quad \mathbb{E}[D_{\bt}(T_{n_{0}}^{(S_{2})})].
	\end{equation}
	To see this,  suppose by symmetry that $|S_1| \leq |S_2|$ and set $S_1' = T_{|S_2|}^{(S_1)}$. 
	Thus $S_1'$ is a random tree when $|S_1| < |S_2|$. If we take $\bt = S_2$ 
	with labels $\ell(u) = \deg(u)$, then, for every tree $T$  with $|T| = |S_2|$, 
	we have $D_{\bt}(T) = D_{\bt}(S_2) \cdot \ind_{  \{T = S_2\}}.$
	Consequently, for this particular value of $\bt$, 
	$$ \Es{ D_{\bt}(S_{1}') } = D_{\bt}(S_2) \cdot \PP [S_{1}' = S_2 ].$$
	When $|S_1| = |S_2|$ the above probability is $0$.
	When $|S_1| < |S_2|$ it may easily be checked that $S_1'$ is non-deterministic 
	(here it is essential that $|S_1| \geq 3$), 
	hence  the probability above is strictly less than $1$. 
	In both cases~\eqref{eq:bt} holds for this choice of $\bt$. 
	
	Let $\bt$ be a minimal {(for the partial order $ \preccurlyeq$)} decorated tree for which~\eqref{eq:bt} holds. 
	Then $ \mathbb{E}[ D_{\btp}(T_{n_{0}}^{(S_{1})})] =  \mathbb{E}[ D_{\btp}(T_{n_{0}}^{(S_{2})})]$ for all $\btp \prec \bt$ 
	and it follows that $$\Es {M^{(S_{1})}_{\bt}(n_{0}) } \neq \Es{ M^{(S_{2})}_{\bt}(n_{0})},$$
	where $	M^{(S_1)}_{\bt}$ and $M^{(S_2)}_{\bt}$ are martingales as in Proposition~\ref{prop:martingales}.  To simplify notation, set $M_{1}(n)=M^{(S_{1})}_{\bt}(n)$ and $M_{2}(n)=M^{(S_{2})}_{\bt}(n)$. For $n \geq n_{0}$, we may bound the distance in total variation between 
	$T_{n}^{(S_{1})}$ and $T_{n}^{(S_2)}$ as follows
	(see for instance \cite[p.8] {BMR14v2})
	\begin{eqnarray*}
		d_{TV}\left(T_{n}^{(S_{1})},T_{n}^{(S_{2})}\right)  \geq  d_{TV}\left(M_{1}(n),M_{2}(n)\right)  
		\geq   \frac{ \left( \Es {M_{1}(n)- M_{2}(n)}\right)^{2}}
		{2\left( \mathrm{Var}\left(M_{1}(n)\right)+  \mathrm{Var} \left(M_{2}(n)\right)\right)
		+ \left( \Es {M_{1}(n)- M_{2}(n)}\right)^{2}}.
	\end{eqnarray*}
	Since $M_{1}$ and $M_{2}$ are martingales, we have 
	$ \Es {M_{1}(n)}- \Es {M_{2}(n)} = \Es {M_{1}(n_0)}- \Es {M_{2}(n_0)}  \neq 0$
	and $\mathrm{Var}(M_{1}(n))+  \mathrm{Var} (M_{2}(n))$ 
	is bounded as $n \rightarrow \infty$ since the two martingales are bounded in $\mathbb{L}^2$. 
	Thus the quantity $d_{TV}\left(T_{n}^{(S_{1})},T_{n}^{(S_{2})}\right)$ is uniformly bounded away from $0$ as $n \to \infty$ as desired. 
\end {proof}

\subsection{The recurrence relation}
In this section, we present a recurrence relation for the conditional expectations of $ D_{ \bt}(T_{n}^{(S)})$. This relation is the key to Theorem~\ref{thm:metric} since it is used to build the martingales of Proposition~\ref{prop:martingales} and get moment estimates on them.

\begin{lemma}\label {lem:espcond}
	There exists a family of nonnegative real numbers $\{c(\bt,\btp) : \btp \prec\bt\}$ such that, 
	for every seed $S$, every decorated tree $\bt$ with $w(\bt) >1$ and every $n \geq |S|$ we have 
	\begin{equation}\label{eq:rec}
		\Es {D_{ \bt} \big( T_{n+1}^{(S)} \big)  \big | \mathcal{F}_{n} }
		= \left( 1 + \frac{w(\bt)}{2n-2} \right)  D_{ \bt} \big( T_{n}^{(S)} \big) +
		\frac{1}{2n-2} \sum_{\btp \prec \bt} c(\bt, \btp) D_{\btp} \big( T_{n}^{(S)} \big).
	\end{equation}
	When $\tau = \text{\ding{172}}$ we have $D_{\text{\ding{172}}} ( T_{n}^{(S)}) = 2n-2$.
\end{lemma}

\begin{proof}
\newcommand{\weight}{\mathcal{W}}

Fix a tree $S$ with $|S| \geq 2$ and $n >|S|$. To simplify notation, we omit the dependence on $S$ and write $T_{n}$ instead of $T^{(S)}_{n}$. 
It will be clear by construction that the coefficients $c(\bt, \btp)$ do not depend on $S$.  We have already noticed that when $ \bt = \text{\ding{172}}$, $D_{ \bt} ( T_{n})$ is  the total degree of $T_{n}$, which is indeed  $2n-2$. 

Now fix a decorated tree $\bt$ with $w(\bt) \geq 2$. We denote by $u_{n+1}$ the vertex present in $T_{n+1}$ but not in $T_n$, and by $v_n$ its neighbour in $T_{n}$. We write the set all embeddings $\phi : \tau \rightarrow T_{n+1}$  as the disjoint union of the set of those using only vertices of $T_n$, denoted by $ \mathcal{E}_{n}$, and the set of those using the new vertex $u_{n+1}$, which is denoted by  $\mathcal{E}_{n+1} \backslash \mathcal{E}_{n}$.
If $T$ is a tree and $\phi : \tau \to T $ is an embedding, 
we write $\weight_{\phi}(T) = \prod_{u \in \tau} \left[\deg_{T}\phi(u) \right]_{ \ell(u)}$.

Let us evaluate $\Es {D_{ \bt} \big( T_{n+1} \big) \ \big | \ \mathcal{F}_{n} }$.
Since we work conditionally on $\mathcal{F}_n$, we may consider $T_n$ as being fixed. 
Then 
$D_{ \bt} \big( T_{n+1} \big)  =  \sum_{\phi \in \mathcal{E}_{n+1} } \weight_{\phi}(T_{n+1}),$ 	
and we split the sum into two, depending on whether $\phi \in \mathcal{E}_n$ 
or $\phi \in \mathcal{E}_{n+1} \setminus \mathcal{E}_n$. 
First of all, it is a simple matter to check that for every  $\ell,d \geq 1$,
$$ \begin{array}{lrcl}(\star) & [ d+1]_{\ell} &=& [d]_{\ell} + \ell \cdot [d]_{\ell- 1}, \\ 
 (\star \star) & d \cdot [d]_{\ell - 1} &=&  [d]_{\ell} + (\ell -1) \cdot [d]_{\ell -1},\\
 (\star\star \star) \quad & d \cdot [d +1]_{\ell}
&=&[d]_{\ell + 1}
+ 2  \ell \cdot [d]_{\ell}
+ \ell  (\ell - 1) \cdot [d]_{\ell -1}.  \end{array}$$
First assume that  $\phi \in \mathcal{E}_n$. 
Since $ \mathrm{deg}_{T_{n+1}}(v_{n}) = \mathrm{deg}_{T_{n}}(v_{n})+1$, it follows that
\begin{eqnarray*}
	&& \Es{ \weight_{\phi}(T_{n+1}) \ | \ \mathcal{F}_n } \\
	 & \underset{(\star)}{=}& \weight_{\phi}(T_n) +
		\Es{ 
		\sum_{w \in \bt} \ind_{  \{\phi(w) = v_n\}}
		\cdot \ell(w) \cdot \left[\deg_{T_n}\phi(w) \right]_{\ell(w)-1}
		\cdot \prod_{w' \in \bt \setminus \{w\}} \left[\deg_{T_n}(\phi(w')) \right]_{\ell(w')}\,\Big| \ \mathcal{F}_n } \\
	&=& \weight_{\phi}(T_n) + 
		\sum_{w \in \bt} 
		\frac{\deg_{T_n}\phi(w)}{2n-2} 
		\cdot \ell(w) \cdot \left[\deg_{T_n}\phi(w) \right]_{\ell(w)-1}
		\cdot \prod_{w' \in \bt \setminus \{w\}} \left[\deg_{T_n}(\phi(w)) \right]_{\ell(w')} \\
	&\underset{(\star \star)}{=}& \left(1 + \frac{w(\tau)}{2n-2}\right)  \cdot \weight_{\phi}(T_n) +
		\sum_{w \in \bt} \frac{\ell(w)  (\ell(w) - 1)}{2n-2} \cdot
		\left[\deg_{T_n}\phi(w) \right]_{\ell(w) -1 }
		\cdot \prod_{w' \in \bt \setminus \{w\}} \left[\deg_{T_n}\phi(w) \right]_{\ell(w')}  \\	
	&=& \left(1 + \frac{w(\tau)}{2n-2}\right)  \cdot \weight_{\phi}(T_n) +
		\sum_{w \in \bt} \frac{\ell(w)(\ell(w) - 1)}{2n-2} \cdot \weight_{\phi_w}(T_n),	
\end{eqnarray*}
where $\phi_w$ is the embedding equal to $\phi$ of the decorated tree $\bt_w$ identical to $\bt$
except for the label of $w$ which is $\ell_{\bt_w} (w) = \ell_{\bt} (w) -1$. 
Note that such trees appear in the expression only if $\ell_{\bt_w} (w) > 0$. 
When $\phi$ runs through the embeddings of $\bt$ in $T_n$, 
$\phi_w$ runs exactly through the embeddings of $\bt_w$ in $T_n$. 
Thus 
\begin{eqnarray}\label{eq:Wn}
		\mathbb{E}\left[ \left.\sum_{\phi \in \mathcal{E}_{n} } \weight_{\phi}(T_{n+1}) \,\right|\, \mathcal{F}_n \right]
		&=&
		\left(1 + \frac{w(\tau)}{2n-2}\right)  D_{\bt}(T_n) +
		\sum_{w \in \bt} \frac{\ell(w)(\ell(w) - 1)}{2n-2} \cdot D_{\bt_w}(T_n).
\end{eqnarray}
Notice that the trees $\tau_{w}$ for $w \in \tau$ may not be distinct but  all have the property $ \tau_{w} \prec \tau$.

Denote by $L$ the set of all leaves (i.e.\,vertices of $ \tau$ of degree $1$) $w \in \tau$ such that $\ell(w) = 1$. If  $\phi \in \mathcal{E}_{n+1} \setminus \mathcal{E}_n$, consider  $w \in \tau$  such that $ \phi(w)=u_{n+1}$.  Note that if $w \not \in L$, then necessarily $\weight_{\phi}(T_{n+1}) = 0$. Since $\tau \neq \text{\ding{172}}$, we may assume that $|\tau| \geq 2$. If $w \in L$, we denote by $ \mathcal{E}_{w}$ the set of all embeddings  $\phi \in \mathcal{E}_{n+1} \setminus \mathcal{E}_n$ such that $ \phi(w)=u_{n+1}$. Now fix $w  \in L, \phi \in \mathcal{E}_{w}$ and let $a$ be the neighbour of $w$ in $ \tau$. Since $[ \deg_{T_{n+1}}( \phi(w))]_{ \ell(w)}=1$ and $ \phi : \tau \backslash  \{w\} \rightarrow T_{n}$ restricted to $\tau \backslash  \{w\}$ is an embedding, we have
\begin{eqnarray*}
	\Es{ \weight_{\phi}(T_{n+1})\ | \ \mathcal{F}_n } 
	 &=& \Es{ \ind_{  \{\phi(a) = v_n\}}
		\cdot \prod_{w' \in \bt \setminus {w}} { \left[\deg_{T_{n+1}}(\phi(w')) \right]_{\ell(w')}} \,\Big| \mathcal{F}_n } \\
	&=& \frac{\deg_{T_n}\phi(a)}{2n-2} \cdot
			{ \left[\deg_{T_{n}}\phi(a) + 1 \right]_{\ell(a)}}
			\cdot \prod_{w' \in \bt \setminus \{w, a\}} { \left[\deg_{T_{n}}(\phi(w')).\right]_{\ell(w')}} 
\end{eqnarray*}

Write $\bt_w^1$, $\bt_w^2$ and $\bt_w^3$ for the 
decorated trees obtained from $\bt$ by removing $w$ and respectively 
increasing by one the label of $a$, leaving it unchanged and 
decreasing it by one. 
Let  $\phi_w^1$, $\phi_w^2$ and $\phi_w^3$ be the respective embeddings 
of $\bt_w^1$, $\bt_w^2$ and $\bt_w^3$ in $T_{n}$
obtained by restricting $\phi$ to $ \tau \backslash  \{w\}$. 
Then, using the previous display and $(\star\star\star)$, we obtain
\begin{eqnarray*}
	\Es{ \weight_{\phi}(T_{n+1})\ | \ \mathcal{F}_n }
	= \frac{1}{2n-2}\left( \weight_{\phi_w^1}(T_{n}) + 
	2\ell(a) \cdot \weight_{\phi_w^2}(T_{n}) +
	\ell(a)(\ell(a) -1 ) \cdot \weight_{\phi_w^3}(T_{n}) \right).
\end{eqnarray*}
Now note that, for fixed $w \in L$, as $\phi$ runs through $ \mathcal{E}_{w}$, the embeddings
$\phi_w^1$, $\phi_w^2$ and $\phi_w^3$ run respectively through all the embeddings of  $\bt_w^1$, $\bt_w^2$ and $\bt_w^3$ in $T_n$. 
Thus, summing over all $w \in L$, we obtain
\begin{eqnarray}\label{eq:Wnn+1}
	\sum_{\phi\in \mathcal{E}_{n+1} \setminus \mathcal{E}_{n} }	
	\Es{ \weight_{\phi}(T_{n+1}) \ | \ \mathcal{F}_n } 
	 = \frac{1}{2n-2}\sum_{w \in L}  \left(D_{\bt_w^1}(T_{n}) + 
	2\ell(a) \cdot D_{\bt_w^2}(T_{n}) +
	\ell(a)(\ell(a) -1 ) \cdot D_{\bt_w^3}(T_{n}) \right).
\end{eqnarray}

After adding~\eqref{eq:Wn} and~\eqref{eq:Wnn+1}, one gets an expression for 
$\Es{D_{\bt}(T_{n+1}) \ | \ \mathcal{F} _n}$ as a function of $D_{\bt}(T_{n})$, $D_{\bt_w}(T_{n})$,
$D_{\bt_w^1}(T_{n})$, $D_{\bt_w^2}(T_{n})$ and $D_{\bt_w^3}(T_{n})$, 
with coefficients depending on $\bt$ only.
Finally, we mention that $\bt_w, \bt_w^1, \bt_w^2, \bt_w^3$ are all smaller than $\bt$ for the order $\prec$. 
\end{proof}

We now sketch another possible argument to prove Lemma~\ref{lem:espcond} relying  on decorated embeddings, which were defined just after the statement of Proposition \ref{prop:martingales}. We mention this approach since a similar one will be used later. 

First note that $T_{n+1}$ contains two more corners than $T_n$:
one around the vertex to which the new edge is grafted, 
and one around the new vertex which is added in the transition from $T_{n}$ to $T_{n+1}$. 
Call these corners respectively $c_n'$ and $c_n''$
and let $c_n$ be the corner of $T_n$ in which the additional edge of $T_{n+1}$ is grafted. 
Then $c_n$ corresponds to one of the neighbouring corners of $c_n'$ in $T_{n+1}$. 
The $D_{\bt}(T_{n+1})$ decorated embeddings of $\bt$ in $T_{n+1}$ may be split between 
those using at least one of the corners $c_n'$, $c_n''$, and those using none of them. 
There are $D_{\bt}(T_n)$ embeddings of the second type. 
With each decorated embedding $\phi$ of the first type, 
associate an embedding $\phi_{\bs}$ of a decorated tree $\bs \preccurlyeq \bt$ in $T_n$
obtained by conserving the arrows associated by $\phi$ with corners common to $T_n$ and $T_{n+1}$
and adding an arrow pointing to $c_n$ (if such an arrow did not already exist). 

Depending on which of the corners $c_n, c_n', c_n''$ are used by $\phi$, $\bs$ takes different values. 
Thus $D_{\bt}(T_{n+1}) - D_{\bt}(T_{n})$ may be expressed as a linear combination 
of numbers of decorated embeddings of trees $\bs \preccurlyeq \bt$ in $T_{n}$ with an arrow in the corner $c_n$.
But since $c_n$ is uniform among the corners of $T_n$, 
in expectation these numbers are $\frac{1}{2n-2}D_{\bs}(T_n)$, which leads to 	\eqref{eq:rec}. 

The proof of Lemma~\ref{lem:embed3} illustrates in more detail the use of these ideas.

\subsection{Moment estimates}

Relying on Lemma~\ref{lem:espcond}, we now establish moment estimates on the number of decorated embeddings, 
which will be used to check that the martingales of Proposition \ref{prop:martingales} are bounded in $ \L^2$. 
In the following,  if $(a_{n})_{n \geq 0}$ and $(b_{n})_{n\geq 0}$ are two sequences of  real numbers, 
we write $a_{n} \ll b_{n}$  if there exist $c>0$ and $\gamma \in \R$ such that 
$|a_{n}| \leq c \cdot  \log(n)^ \gamma \cdot |b_{n}|$ for $n$ large enough. 

In this section, we fix a tree $S$ with $|S|\geq 2$, and write $T_n$ for $T_n^{(S)}$ to simplify notation.

\begin{corollary}\label{cor:embed} 
	Let $\bt$ be a decorated tree with $w(\bt) >1$. Then, for every seed $S$, we have $$ n^{w(\bt)/2} \ll \Es{D_{\bt}(T_{n})} \ll n^{w(\bt)/2}.$$
\end{corollary}
\begin{proof}[Proof of Corollary~\ref{cor:embed}]The first bound is immediate because Lemma~\ref{lem:espcond} implies that, for $\bt$ with $w(\bt) > 1$, 
	$$ \Es {D_{ \bt} \big( T_{n} \big)} \geq
	C \cdot \prod_{j=|S|}^{n-1}\left( 1+ \frac{w(\bt)}{2j-2} \right) \geq C' \cdot n^{w(\bt)/ 2},$$
	for constants $C,C' >0$ depending on $S$ and $\bt$.
	
	We prove the second bound by induction on $\bt$ (for the partial order $ \preceq$). 
	Fix $\bt$ with $w(\bt) >1$ and assume that
	$\Es{D_{\bt'}(T_{n})} \ll n^{w(\bt')/2}$ for every $\bt' \prec \bt$ with $w(\bt') >1$. Since $D_{\text{\ding{172}}}(T_{n}) = 2n-2$, it follows that $\Es{D_{\bt'}(T_{n})} \ll n^{w(\bt)/2}$ for every $\bt' \prec \bt$. Then, by Lemma~\ref{lem:espcond}, there exist constants $C, \gamma>0$ such that 
	$$ \Es {D_{ \bt} \big( T_{n+1} \big)   } 
	\leq  \left( 1+ \frac{w(\bt)}{2n-2} \right)  \cdot \Es{D_{ \bt} \big( T_{n} \big)} +
	C \cdot ( \log n)^ {\gamma} \cdot n^{\frac{w(\bt)}{2} -1} .
	$$
	It is then a simple matter to show that this implies $\Es{D_{ \bt} \big( T_{n} \big)}  \ll n^{w(\bt)/2}$.
\end{proof}

We now turn to second  moment estimates on $D_{\bt}(T_n)$ which will be useful in the proof of Proposition~\ref{prop:martingales}.

\begin{lemma}\label{lem:embed3}	 
	Let $ \bt$ be a decorated tree with $w(\bt) >1$. Then
	\begin{enumerate}[(i)]
 	\item $ \Es{D_{\bt}(T_{n})^2} \ll n^{w(\bt)}$,
 	\item $ \Es{ \big(D_{\bt}(T_{n+1})-D_{\bt}(T_{n}) \big)^2} \ll n^{-3/2} \cdot n^{w(\bt)}$.
  	\end{enumerate}
\end{lemma}

To establish these results, we will need to estimate the number of embeddings in $T_n$ of pairs of decorated trees. 
If $\bt$ and $ \btp$ are decorated trees, set
\begin{equation}\label{eq:Dtt} 
	D_{\bt,\btp}(T)= \sum_{ \phi} \prod_{u \in \tau \sqcup \tau'} [\phi(u)]_{ \ell(u)},
\end{equation} 
where the sum is taken over all injective graph homomorphisms from $ \tau \sqcup \tau'$ to $T$ 
(in particular, $ \phi(u) \neq \phi(u')$ if $u \in \tau$ and $ u' \in \tau'$).

\begin{lemma}\label{lem:embed2}  
	Let $ \bt,\bt'$ be two decorated trees. 
	Then $\Es{D_{\bt,\bt'}(T_{n})} \ll \Es{D_{\bt}(T_{n})} \cdot \Es{D_{\bt'}(T_{n})}$.
\end{lemma}

\begin{proof}[Sketch of proof of Lemma~\ref{lem:embed2}]  
	The proof follows the same lines as that of Lemma~\ref{lem:espcond} and Corollary~\ref{cor:embed}. 
	For this reason, we only lay out the main steps without giving additional detail. 
	As in Lemma~\ref{lem:espcond}, one starts by writing a recurrence relation for $\Es{D_{\bt,\bt'}(T_{n})}$ of the following form: 		
	\begin{eqnarray*}
		\Es {D_{ \bt} \big( T_{n+1}\big)  \big | \mathcal{F}_{n} }
			&=& \left( 1 + \frac{w(\bt) + w(\bt')}{2n-2} \right)  D_{ \bt, \bt'} \big( T_{n}\big) \\
			&& + \frac{1}{2n-2} \left( \sum_{\bs\prec \bt} c(\bt, \bt', \bs) D_{\bs,\bt'} \big( T_{n}\big) +
			  \sum_{\bs' \prec \bt'} c'(\bt, \bt', \bs') D_{\bt, \bs'} \big( T_{n}\big) \right),
	\end{eqnarray*}
	for certain nonnegative real numbers $c(\bt, \bt', \bs), c'(\bt, \bt', \bs) $ and $n \geq |S|$. 
	We stress that in the previous equation, the decorated trees $\bs$ and $\bs'$ may also take the value $\emptyset$,
	in which case $D_{\bs,\bt'} \big( T_{n}\big)$ and $D_{\bt, \bs'} \big( T_{n}\big) $
	are equal to respectively $D_{\bt'} \big( T_{n}\big)$ and $D_{\bt} \big( T_{n}\big) $. 
	The same inductive argument as that of Corollary~\ref{cor:embed} leads to the conclusion. 
\end{proof}

\begin{proof}[Proof of Lemma~\ref{lem:embed3}] 
To simplify the proof, we use a planar embedding of $T_{n}$ and the interpretation of $D_{\bt}(T_{n})$ 
as the number of decorated embeddings of $ \bt$ in $T_{n}$, as explained after the statement of Proposition~\ref{prop:martingales}. Let $ \bt'$ be a disjoint copy of $ \bt$. By definition, a decorated map $ \bphi : \bt  \cup \bt' \rightarrow T_{n}$ is a map such that both $ \bphi |_{ \bt}$  and $ \bphi |_{ \bt'}$ are decorated embeddings. We insist on the fact that $ \bphi$ is not necessarily injective.   If $ \bphi$ is a decorated embedding or a decorated map, $ \phi$ will denote the map without the choice of corners associated with arrows.

For the first assertion, observe that 
$D_{\bt}(T_{n})^2$ is the number of decorated maps $ \bphi : \bt \cup \btp \rightarrow T_{n}$.
We denote by $ \mathcal{E}^{1}_{\bt}(T_{n})$ 
the set of all such decorated maps with $ \phi( \tau) \cap \phi( \tau')= \emptyset$ 
(as in the definition of $D_{\bt, \bt}(T_{n})$), 
and by $ \mathcal{E}^{2}_{\bt}(T_{n})$ the set of all such decorated maps with 
$\phi( \tau) \cap \phi( \tau') \neq \emptyset$. 
The cardinality of $ \mathcal{E}^{1}_{\bt}(T_{n})$ is $D_{\bt, \bt}(T_{n})$, and Lemma~\ref{lem:embed2} applies.

If $ \bphi \in \mathcal{E}^{2}_{\bt}(T_{n})$ is a decorated map, we may associate with $ \bphi$ a decorated embedding $\bphi '$ 
of a decorated tree $\bs_{\bphi}$ obtained by overlapping two copies of $\bt$. 
More precisely, let  $\mathcal{U}_2(\bt)$ be the set of all decorated trees 
which may be obtained by identifying a non-empty subset of elements (i.e. of vertices, edges and arrows) of $\bt$ and $\btp$. 
The embedding associated with $\bphi$ is given by the images of $\bt$ and $\bt'$ in $T_n$ via $\bphi$
(in particular $\bs_{\bphi}$ is the union of the images of $\bt$ and $\bt'$), see Figure \ref{fig:embed}.

Note that the function $\bphi \mapsto  (\bs_{\bphi},\bphi')$ defined above is not one to one, 
since an element of $\mathcal{U}_2(\bt)$ may be obtained in several ways  by overlapping $\bt$ and $\btp$. 
However, it is easy to see that there exists a constant $C( \bt)>0$ 
such that any decorated tree $ \bs \in \mathcal{U}_2(\bt)$ and any embedding $\bphi'$ of $\bs$ in $T_n$
is associated with at most $C( \bt)$ decorated maps $\bphi$. 
We may therefore conclude that
\begin{eqnarray*}
D_{\bt}(T_{n})^2
= \# \mathcal{E}^{1}_{\bt}(T_{n})+ \# \mathcal{E}^{2}_{\bt}(T_{n})
& \leq & D_{\bt, \bt}(T_{n}) + C(\bt) \cdot \sum_{\bs \in \mathcal{U}_2(\bt)} D_{\bs}(T_{n}).
\end{eqnarray*}
Observe that $w(\bs) \leq 2 \cdot w(\bt)$ for every $\bs \in \mathcal{U}_2(\bt)$. 
Lemma~\ref{lem:embed2}, Corollary~\ref{cor:embed} and the fact that $\mathcal{U}_2(\bt)$ is a finite set imply the desired bound. 

For the second assertion, we work conditionally on $ \mathcal{F}_{n}$.  As in the discussion after the proof of Lemma~\ref{lem:espcond},
let $u_{n+1}$ be the vertex added to $T_n$ in the transition from $T_{n}$ to $T_{n+1}$,
and $c_n''$ be the corner adjacent to $u_{n+1}$ in $T_{n+1}$ (since $u_{n+1}$ is a leaf of  $T_{n+1}$, there is only one corner adjacent to it). 
Also let $c_n$ be the corner of $T_n$ in which the additional edge of $T_{n+1}$ is grafted,
and denote by $c_n$ and $c_n'$ the two corners of $T_{n+1}$ resulting from  splitting  $c_n$.  
Finally let $v_n$, be the vertex adjacent to $c_n$. We refer to Fig.~\ref{fig:embed} for an example.
\begin{figure}[!h]
 \begin{center}
 \includegraphics[width=0.9 \linewidth]{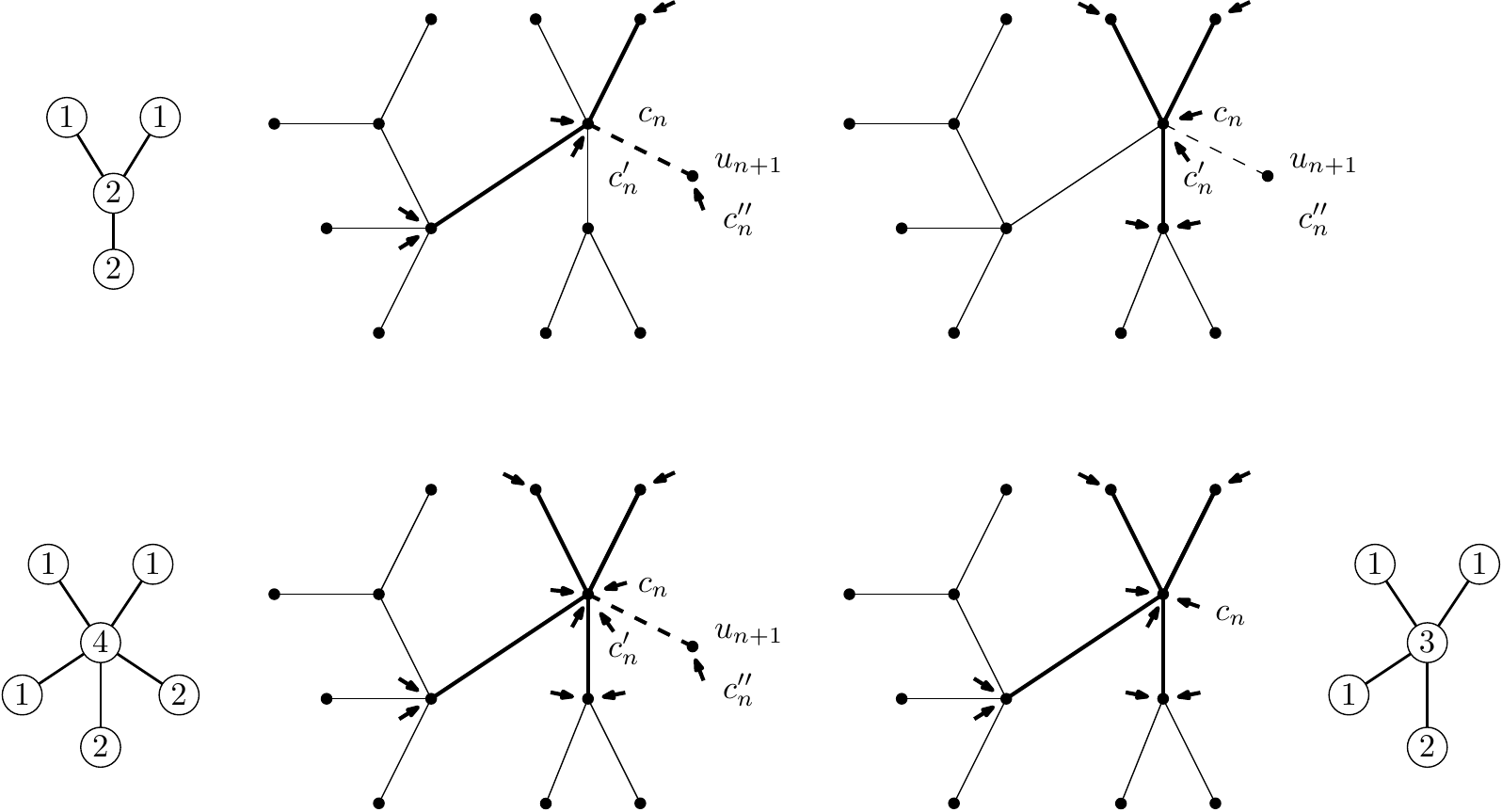}
 \caption{\emph{Top diagram:} Two decorated embeddings $ \bphi_{1}$ and $ \bphi_{2}$ of the decorated tree $\bt$ in $T_{n+1}$.
 Both embeddings use corners of $T_{n+1}$ not present in $T_n$, hence contribute to $D_{\bt}(T_{n+1})-D_{\bt}(T_{n})$.
 \emph{Bottom diagram:} The overlapping of the two decorated embeddings $ \bphi_{1}$ and $ \bphi_{2}$ of $\bt$ induces a decorated embedding of the decorated tree $\bs$ depicted on the left. 
 To $\bs$ we associate the embedding of $\bs'$ in $T_n$ (right diagrams).}
 \label{fig:embed}
 \end{center}
 \end{figure}
 
Note that
$D_{\bt}(T_{n+1})-D_{\bt}(T_{n})$ is the number of decorated embeddings of $\bt$ in $T_{n+1}$ 
that use at least one of the corners $c_n', c_n''$. 
Similarly, $\left(D_{\bt}(T_{n+1})-D_{\bt}(T_{n})\right)^2$ 
is the number of decorated maps $ \bpsi : \bt \cup \btp \rightarrow T_{n+1}$ such that  $ \bpsi |_{ \bt}$  and $ \bpsi |_{ \btp}$ 
both use at least one of the corners $c_n', c_n''$. 
To simplify notation, denote by $ \mathcal{E}_{\bt}(T_{n+1})$ the set of all such decorated maps.
Fix $ \bpsi \in \mathcal{E}_{\bt}(T_{n+1})$.  Since $w ( \bt)>1$, we have $v_{n} \in \psi( \tau)$ and $v_{n} \in \psi( \tau')$. 
As in the proof of (i), we may associate with the decorated map $ \bpsi$ a decorated embedding 
$ \bphi : \bs \rightarrow T_{n+1}$ of a decorated tree $\bs \in \mathcal{U}_2(\bt)$. 
We shall now furthermore associate with $\bphi$ a decorating embedding $\bphi': \bs' \rightarrow T_{n}$ of a modified decorated tree $\bs'$ in $T_{n}$.

Let $w$ be the vertex of $\bs$ such that $ \phi(w)=v_{n}$.
Define the modified decorated tree $\bs'$  by altering $\bs$ as follows. 
Remove from $\bs$ the vertex mapped by $\bphi$ to $u_{n+1}$ as well as the arrow pointing to it, if such a vertex exists.
Remove in addition the arrows of $\bs$ mapped by $\bphi$ to $c_n$ or $c_n'$. 
Finally add an arrow $ \vec{a}$ to $\bs'$ pointing to $w$. 
The decorated embedding $\bphi' : \bs' \rightarrow T_{n}$ is defined 
to be equal to $\bphi$ on $\bs \cap \bs'$, 
and maps the arrow $\vec{a}$ of $\bs'$ on the corner $c_n$. 

To sum up, with every decorated map $\bpsi \in \mathcal{E}_{\bt}(T_{n+1})$, 
we have associated a decorated tree $ \bs'$ with a distinguished arrow $\vec{a}$ and a decorated embedding $\bphi$ of $\bs'$ in $T_{n}$.
Moreover, we have done this in such a way that $\bphi(\vec{a})=c_{n}$. 
In addition,  $\bs'$ satisfies $w(\bs') \leq 2  w(\bt) - 1$ and $|\bs'| \leq 2|\bt| -1 $.
As in (i), this association is not injective but the number of pre-images of any given image may be bounded by a constant $C(\bt)>0$. 
Hence 
\begin{equation*}
\big(D_{\bt}(T_{n+1})-D_{\bt}(T_{n})\big)^2
\leq
\sum_{\substack{\bs' ;  |\bs'| < 2|\bt|, \\ w(\bs') < 2  w(\bt)}} 
\sum_{ \vec {a} \in  \bs'}
\sum_{ \bphi : {\bs'} \rightarrow T_{n} }
 C(\bt) \cdot  \mathbbm {1}_{ \{ \bphi( \vec {a})=c_{n}\} },
\end{equation*}
where the first sum is taken on decorated trees  $\bs'$ with the displayed constraints. 
Since  the corner $c_n$ is uniform in $T_n$, we obtain 
$$
\Es{\big(D_{\bt}(T_{n+1})-D_{\bt}(T_{n})\big)^2 \big| \mathcal{F}_n}
\leq
\sum_{\substack{\bs' ;  |\bs'| < 2|\bt|, \\ w(\bs') < 2  w(\bt)}} 
\sum_{ \vec {a} \in  \bs'}
\frac{C(\bt)}{2n-2} \cdot D_{\bs'}(T_n)
\ll n^{w(\bt) - 3/2}.
$$
For the last estimate, we have used the fact that  $D_{\bs'}(T_n) \ll n^{w(\bt) - 1/2}$ for all values of $\bs'$ 
(see Corollary~\ref{cor:embed}) and that the number of terms in the sum is bounded in terms of $\bt$ only. 
Taking the expectation of the expression above leads to the desired inequality. 
\end{proof}

\subsection{Constructing martingales}

We now use the recurrence relation~\eqref{eq:rec} in order to construct the martingales of Proposition~\ref {prop:martingales}. 
It may be instructive for the reader to compute the martingale $M^{(S)}_{\bt}(n)$ using~\eqref{eq:rec}
in some simple cases, for instance when $\bt$ is formed of a single vertex with label $2$ or $3$, 
or two vertices of label $1$ linked by an edge.

\begin{proof}[Proof of Proposition~\ref{prop:martingales}] 
 Fix a seed tree $S$ with $|S|=n_{0}\geq 2$. For a decorated tree $\bt$ and $n \geq 2$, set
$$\alpha^{\bt}_n =  \prod_{j = 2}^{n-1} \left(1 + \frac{w(\bt)}{2j-2}\right)^{-1}, \textrm{ when } w(\bt) > 1; 
\qquad  \alpha^{\bt}_n = \frac{1}{2n-2},  \textrm{ when } w(\bt)=1.$$
In particular, if $w( \bt)>1$, observe that $ n^{- w(\bt)/2} \ll \alpha^{\bt}_n \ll  n^{- w(\bt)/2}  $ .
For a sequence of real numbers $(a_n)_{n\geq 0}$  set $\Delta_n a = a_{n+1} - a_n$ for $n \geq 0$. 

We start by constructing by induction (on the order $\prec$ on decorated trees)  
coefficients $\{a_n( \bt, \bt'):  \bt' \prec \bt, n \geq n_0\}$ such that   
\begin{eqnarray} 
&& a_n( \bt, \bt') \ll 1, \qquad    \qquad \Delta_n a( \bt, \bt') \ll 1/n  \label{eq:petit} \\ 
\label{eq:defmart}
	\mbox{ and } \qquad M^{(S)}_{\bt}(n) &=& 
	\alpha_n^{\bt} \left( D_{\bt}(T_{n}^{(S)}) {-} \sum_{\bt' \prec \bt} a_n(\bt, \bt') \cdot D_{\bt'}(T_{n}^{(S)}) \right) 
	\qquad \mbox{is a martingale}.
 \end{eqnarray}
We emphasize that by construction, the coefficients $a_n(\bt, \bt')$ will not depend on~$S$ (this is essential).  
To simplify notation, we write $T_n$ and $M_{\bt}(n)$  for respectively  $T_n^{(S)}$  and $M^{(S)}_{\bt}(n) $. 
First, $ M_{\text{\ding{172}}}(n) = \alpha_n^{\text{\ding{172}}} \cdot D_{\text{\ding{172}}}(T_{n}) = 1$, 
which is clearly a martingale. 

Next, fix a decorated tree $\bt$ with $w(\bt) \geq 2$ 
and assume that the coefficients $a_n( \bs, \bs')$ have been constructed for every $\bs' \prec \bs \prec \bt$ and $n \geq n_{0}$, 
and that they have the desired properties. 
Then we claim that  there exist constants $(b_n(\bs,\bs'): \bs' \prec \bs \prec \bt, n \geq n_{0})$ 
such that $b_n(\bs,\bs')  \ll 1$ and
\begin{equation}\label{eq:inverse}
	D_{\bs}(T_{n}) = \frac{1}{\alpha_n^{\bs}} M_{\bs}(n) + 
	\sum_{\bs' \prec \bs} \frac{b_n(\bs, \bs')}{\alpha_n^{\bs'}} M_{\bs'}(n), \qquad n \geq n_{0}.
\end{equation}
Indeed, define the matrix $(\mathcal{A}_n(\bs,\br))_{\bs, \br \prec \bt}$ taking value
$ - a_n(\bs, \br)$ if $\br \prec \bs$, $1$ if $\bs = \br$ and $0$ otherwise. 
Then, by~\eqref{eq:defmart}, for every $n \geq n_{0}$, we have the following equality of vectors indexed by $\bs \prec \bt$:
$$
\left( \frac{1}{\alpha^{\bs}_{n}} \cdot M_{\bs}(n) \right)_{\bs \prec \bt}  
= \mathcal{A}_n \cdot \left(  D_{\bs}(T_{n}) \right)_{\bs \prec \bt}.
$$
We may write $  \{ \bs : \bs \prec \bt\}=  \{\bs_1, \ldots, \bs_K\}$ in such a way that $\bs_i \prec \bs_j \Rightarrow i < j$. 
In this setting, $\mathcal{A}_n$ is a triangular matrix with values $1$ on the diagonal and all coefficients $ \ll 1$. It follows that $\mathcal{A}_n$ is invertible, and that its inverse shares this same property. 
If we denote by $(b_n(\bs, \bs'))_{ \bs' \prec \bs}$ the above-diagonal entries of the inverse of $\mathcal{A}_n$, we obtain~\eqref{eq:inverse}.

Then Lemma~\ref{lem:espcond} and ~\eqref{eq:inverse} yield, for $n \geq n_{0}$,
\begin{eqnarray*}
	&&\Es{\alpha^{\bt}_{n+1} \cdot D_{ \bt} \big( T_{n+1} \big)  \big | \mathcal{F}_{n} }
	= \alpha^{\bt}_n \left(D_{ \bt} \big( T_{n} \big) +
		\frac{1}{2n-2 + w(\bt)} \sum_{\bt' \prec \bt} c(\bt, \bt') D_{\bt'} \big( T_{n} \big)\right) \\		
	&& \qquad \quad = \alpha^{\bt}_n \cdot D_{ \bt} \big( T_{n} \big) 
		+ \sum_{\bs \prec \bt} 
		\frac{1}{2n-2 + w(\bt)} \cdot \left( c(\bt,\bs) + \sum_{\bs \prec \bt' \prec \bt} c(\bt, \bt') b_n(\bs, \bt')\right)
		\cdot  \frac{\alpha^{\bt}_n}{\alpha_n^{\bs}} \cdot
		 M_{\bs}(n). 
\end{eqnarray*}
Now set\newcommand{\aaa}{\overline{a}} 
\begin{equation}\label{eq:defaaa}
	\aaa_n(\bt, \bs) = 
	 \sum_{j=2}^{n-1} \frac{1}{2j-2 + w(\bt)} \cdot \left( c(\bt,\bs) + 
	      \sum_{\bs \prec \bt' \prec \bt} c(\bt, \bt') b_j(\bt', \bs)\right) \cdot \frac{\alpha^{\bt}_j}{\alpha_j^{\bs}}, 
\end{equation}
for $n \geq n_{0}$, so that
\begin{eqnarray*}
	\Es{\alpha^{\bt}_{n+1} D_{ \bt} \big( T_{n+1} \big)  \big | \mathcal{F}_{n} } 
	 = \alpha^{\bt}_n D_{ \bt} \big( T_{n} \big) 
		+ \sum_{\bs \prec \bt} \big( \aaa_ {  {n+1}}(\bt,\bs) - \aaa_{ {n}}(\bt,\bs)\big)\cdot
		 M_{\bs}(n). 
\end{eqnarray*}
Since $(M_{\bt'}(n))_{n \geq n_{0}}$ is a $( \mathcal{F}_{n})$-martingale for every $\bt' \prec \bt$ by our induction hypothesis, the above implies that 
\begin{eqnarray*}
	M_{\bt}(n) &: = &\alpha^{\bt}_n D_{ \bt} \big( T_{n} \big)  {-}  \sum_{\bt' \prec \bt} \aaa_n(\bt,\bt')  M_{\bt'}(n) \\
	&= & 
	\alpha^{\bt}_n \left( D_{\bt} \big( T_{n} \big)  {-} \sum_{\bt' \prec \bt} \aaa_n(\bt,\bt')  
	\cdot \frac{\alpha_n^{\bt'}}{\alpha^{\bt}_n}  \cdot
		\left( D_{\bt'}(T_{n}) + \sum_{\bs \prec \bt'} a_n(\bt', \bs) \cdot D_{\bs}(T_{n}) \right)\right)\\
	&= & 
	\alpha^{\bt}_n \left( D_{\bt} \big( T_{n} \big)  {-}  
	\sum_{\bs \prec \bt} 
	\left( \aaa_n(\bt,\bs) \cdot \frac{\alpha_n^{\bs}}{\alpha^{\bt}_n}
		 +\sum_{\bs \prec \bt' \prec \bt} 
		\aaa_n(\bt,\bt') \cdot\frac{\alpha_n^{\bt'}}{\alpha^{\bt}_n}
		\cdot a_n( \bt', \bs) 
	\right)\cdot D_{\bs}(T_{n})
	\right).
\end{eqnarray*}
is a $( \mathcal{F}_{n})$ martingale.
Finally, for $\bs \prec \bt$ and $n \geq n_{0}$, set
\begin{equation}\label{eq:defa}
a_n(\bt, \bs) 
	\quad := \quad \aaa_n(\bt,\bs) \cdot \frac{\alpha_n^{\bs}}{\alpha^{\bt}_n}
		+ \sum_{\bs \prec \bt' \prec \bt} \aaa_n(\bt,\bt') \cdot\frac{\alpha_n^{\bt'}}{\alpha^{\bt}_n} \cdot a_n( \bt', \bs). 
\end{equation}
With this notation, it is now clear that the martingale $(M_{\bt}(n))_{n \geq n_{0}}$ defined as above satisfies~\eqref{eq:defmart}. 

Let us now analyse the orders of magnitude of the quantities $a_n(\bt, \bs)$ and $\Delta_{n} a(\bt, \bs)$ in order to establish~\eqref{eq:petit}. We have $\frac{1}{2n-2 + w(\bt)} \ll 1/n$, and for two decorated trees $\bs,  \bs'$ with $w(\bs), w( \bs') >1$ a straightforward computation yields
$$ 
\frac{\alpha_n^{\bs}}{\alpha^{\bs'}_n} \ll n^{\frac{w(\bs') - w(\bs)}{2}}
\quad \text{ and } \quad
\Delta_n \frac{\alpha^{\bs}}{\alpha^{\bs'}} \ll n^{\frac{w(\bs') - w(\bs)}{2} -1}.
$$
In addition, by our induction hypothesis, we have $b_n(\bs, \bt') \ll 1$ for every $\bs \prec \bt' \prec \bt$. 
From~\eqref{eq:defaaa} we get that
$$ 
\Delta_n \aaa(\bt,\bs) \ll n^{\frac{w(\bs) - w(\bt)}{2} -1}
\quad \text{ and } \quad
\aaa_n(\bt,\bs) \ll n^{\frac{w(\bs) - w(\bt)}{2}}
$$ 
for every $\bs \prec \bt$ such that $w(\bs)  { \geq } 2$.  Hence, for $\bs \neq \text {\ding{172}}$, 
\begin{equation} \label{eq:majoration}
\aaa_n(\bt,\bs) \cdot \frac{\alpha_n^{\bs}}{\alpha^{\bt}_n} \ll 1
\quad \text{ and } \quad
\Delta_n \left( \aaa(\bt,\bs) \cdot \frac{\alpha^{\bs}}{\alpha^{\bt}} \right) \ll \frac{1}{n}.
\end{equation}
A separate analysis shows that~\eqref{eq:majoration} also holds when $\bs = \text {\ding{172}}$. 
By the induction hypothesis, we have that  
$a_n( \bt', \bs) \ll 1$ and $\Delta_n a( \bt', \bs) \ll 1/n$ for every $\bs \prec \bt' \prec \bt$.
By combining the previous estimates with Eq.~\eqref{eq:defa} which defines $a_n(\bt, \bs)$, we obtain 
\begin{eqnarray*}
	a_n(\bt, \bs) \ll 1 \quad \text{ and } \quad	\Delta_n a(\bt, \bs) \ll 1/n, \quad \text{ for every } \bt' \prec \bt. 
\end{eqnarray*}
This completes the induction.

Finally, let us now prove that the martingales $ M_{\bt} $ defined by~\eqref{eq:defmart} are indeed bounded in $\L^2$. 
To this end, since the increments of a martingale are orthogonal in $\L^2$, write
$$ \Es{M_{\bt}(n)^2} = \sum_{j=n_0}^{n-1} \Es{ \big(M_{\bt}(j+1)-M_{\bt}(j)\big)^2 } + \Es{M_{\bt}(n_0)^2}.$$
It is clear that $\Es{M_{\bt}(n_0)^2} < \infty$, 
so it is enough to check that 
$$\sum_{n \geq n_{0}} \Es{\big(M_{\bt}(n+1)-M_{\bt}(n)\big)^2} < \infty.$$
By~\eqref{eq:defmart} and the Cauchy-Schwarz inequality, there exists a constant $c > 0$, 
depending only on $\bt$, such that for $n \geq n_{0}$, the quantity $c \cdot \Es{\big(M_{\bt}(n+1)-M_{\bt}(n)\big)^2}$ is bounded from above by
\begin{eqnarray*} 
	&&\Es{ \Big(\Delta_n \big(\alpha^{\bt} \cdot D_{\bt}(T) \big) \Big)^{2}} \\
	&& \qquad\qquad  +\sum_{\bt' \prec \bt}  \left( 
	 \big( a_{n+1}( \bt', \bt)\big)^2 \cdot\Es{ \Big( \Delta_n \big(\alpha^{\bt'}\cdot D_{\bt'}(T) \big) \Big)^{2} }    
	+ \big(\Delta_n a( \bt', \bt)\big)^2 \cdot \Es{ \big(\alpha_n^{\bt'} \cdot D_{\bt'}(T_{n})\big)^{2}} \right) .
\end{eqnarray*}
To bound this quantity, it will be useful to note that, for every $\bs$ with $w(\bs) >1$, 
by Lemma~\ref{lem:embed3} and a straightforward computation,
\begin{eqnarray*}
 	\Es{\Big( \Delta_n \big(\alpha^{\bs} \cdot D_{\bs}(T)\big)\Big)^2} 
	\leq 
	2 \big(\Delta_n\alpha^{\bs}\big)^{2} \cdot\Es{  D_{\bs}(T_{n})^2} 
	+	2 \big(\alpha_{n+1}^{\bs}\big)^{2}\cdot \Es{ \big(\Delta_n D_{\bs}(T)\big)^2}
	\ll n^{-3/2}.
\end{eqnarray*}
In addition, when $\bs = \text {\ding{172}}$, we have $\Delta_n \big(\alpha^{\bs} \cdot D_{\bs}(T)\big) = 0$. 
By combining the previous estimates  with~\eqref{eq:petit}, we finally get that
$$ \Es{(M_{\bt}(n+1)-M_{\bt}(n))^2}  \ll n^{-3/2}. $$
This implies $\sum_{n \geq n_0} \Es{\big(M_{\bt}(n+1)-M_{\bt}(n)\big)^2} < \infty $,
and the proof is complete. 
\end{proof}

\section{Scaling limits of looptrees} \label{section:scaling}
 
In this section, we prove Theorem~\ref{thm:scaling}. 
The mail tool is a coupling between the looptrees of the plane LPAM 
and a certain modification of  binary trees obtained by Rémy's algorithm \cite{Rem85}.
The coupling between the LPAM and Rémy's algorithm has already been noticed in the literature \cite{PRR14}, 
but will be recalled here and extended to looptrees.

\subsection{Coupling with Rémy's algorithm}\label {sec:coupling} \label {sec:dt}

We start by introducing some useful notation. In this section, unless stated otherwise, trees are not considered as embedded in the plane. A tree is binary when all its vertices have degree at most~$3$. If $x,y \in \tau$ are two vertices of a tree $ \tau$, we let $ \llbracket x,y \rrbracket$ be the  geodesic in $ \tau$ between $x$ and $y$. If $x_{0},x_{1}, \ldots, x_{k} \in \tau$ are distinct vertices, we let
$$ \mathsf{Span}( \tau; x_{0}, \ldots,x_{k})= \bigcup_{0 \leq  i,j \leq k} \llbracket x_{i}, x_{j} \rrbracket$$
be the tree spanned by these vertices.  A labeled tree $ \bt= ( \tau; x_{0}, x_{1}, \ldots,x_{n})$ is a pair formed of a tree $ \tau$ and a collection of leaves $(x_0, \ldots, x_n)$ of $ \tau$.

For a labeled tree $ \bt= ( \tau; x_{0}, x_{1}, \ldots,x_{n})$ the gluing $ \mathsf{Glu}( \boldsymbol{ \tau})$ of $ \bt$ is the graph constructed as follows. Set $p_{1}=x_{0}$, and for $2 \leq i \leq n$, let $p_{i}$ be the vertex of $ \mathsf{Span}( \tau ; x_{0},x_{1}, \ldots x_{i-1})$ which is the closest to $x_{i}$ in $\tau$.  Then $\mathsf{Glu}( \boldsymbol{ \tau})$ is by definition the graph obtained from $\tau$ by identifying the vertices $x_{i}$ and $p_{i}$ for every $1 \leq i \leq n$. 
See the second line of Fig.~\ref{fig:coupling} for an illustration.
Formally, the vertices of the graph $\mathsf{Glu}( \boldsymbol{ \tau})$ 
are the equivalence classes of the vertices of $ \tau$ 
for the equivalence relation generated by $x_{i} \sim p_{i}$ for every $1 \leq i \leq n$.  We view $\mathsf{Glu}( \boldsymbol{ \tau})$ as a compact metric space by endowing its vertex set with the graph distance.

\bigskip

Next we present Rémy's algorithm; it is a recursive procedure for building labeled binary trees. Start with the tree $ \mathbf{B}_{1}=(B_{1}; A_{0}, A_{1})$ consisting of a single edge with two leaves labeled $A_{0}$ and $A_{1}$. At every step $n \geq 1$, build $ \mathbf{B}_{n+1}$  from $ \mathbf{B}_{n}$ by picking an edge $e$ of $ \mathbf{B}_{n}$ uniformly at random,  adding a vertex $v$ on $e$ (thus splitting $e$ into two edges) and attaching a new edge to $v$ linking it to a new leaf denoted $A_{n+1}$. Rémy \cite{Rem85} showed that for every fixed $n \geq 1$, the labeled tree $ \mathbf{B}_{n}$ is uniformly distributed over the set of all binary trees with $n+1$ labeled leaves.

Let $ ( {T}_{n}^\multimap)_{n \geq 1}$ be the plane LPAM with seed $\multimap$, as defined in the Introduction. 
Recall that   $\mathsf{Loop}( {T}_{n}^\multimap)$ is the looptree associated with ${T}_{n}^\multimap$.
An important element of the proof of Theorem~\ref{thm:scaling} is the following.

\begin{proposition} \label{prop:coupling}
We have the following joint equality in distribution
$$ ( \mathsf{Loop}( T_{n}^\multimap) ; {n \geq 1})  \quad\mathop{=}^{(d)} \quad (\mathsf{Glu}( \mathbf{B}_{n}) ; {n \geq 1}).$$ 
\end{proposition}

\begin{proof}The growth mechanism of $ \mathsf{Loop}(T_{n}^\multimap)$ is the following: at each step, an edge is selected uniformly at random, split in its middle by adding a new vertex, with attached to it a new loop made of single edge. Let us now turn to the growth mechanism of $ \mathsf{Glu}( \mathbf{B}_{n})$: This graph is the collection of  $n$ loops made by the geodesics starting from $ \mathsf{Span}(B_{n} ; A_{0}, \ldots , A_{i-1})$ and going to $A_{i}$ for $i \in \{1, \ldots , n\}$ which are turned into cycles by identifying their endpoints. Then an edge of $B_{n}$ is selected uniformly at random, and split in its middle by adding a new edge carrying $A_{n+1}$. Then observe that the impact of this splitting on $ \mathsf{Glu}(B_{n})$  is equivalent to the growth procedure of $\mathsf{Loop}(T_{n}^\multimap)$ we have described (see Fig.~\ref{fig:coupling} for an illustration).

\begin{figure}[!h]
 \begin{center}
 \includegraphics[width=0.9 \linewidth]{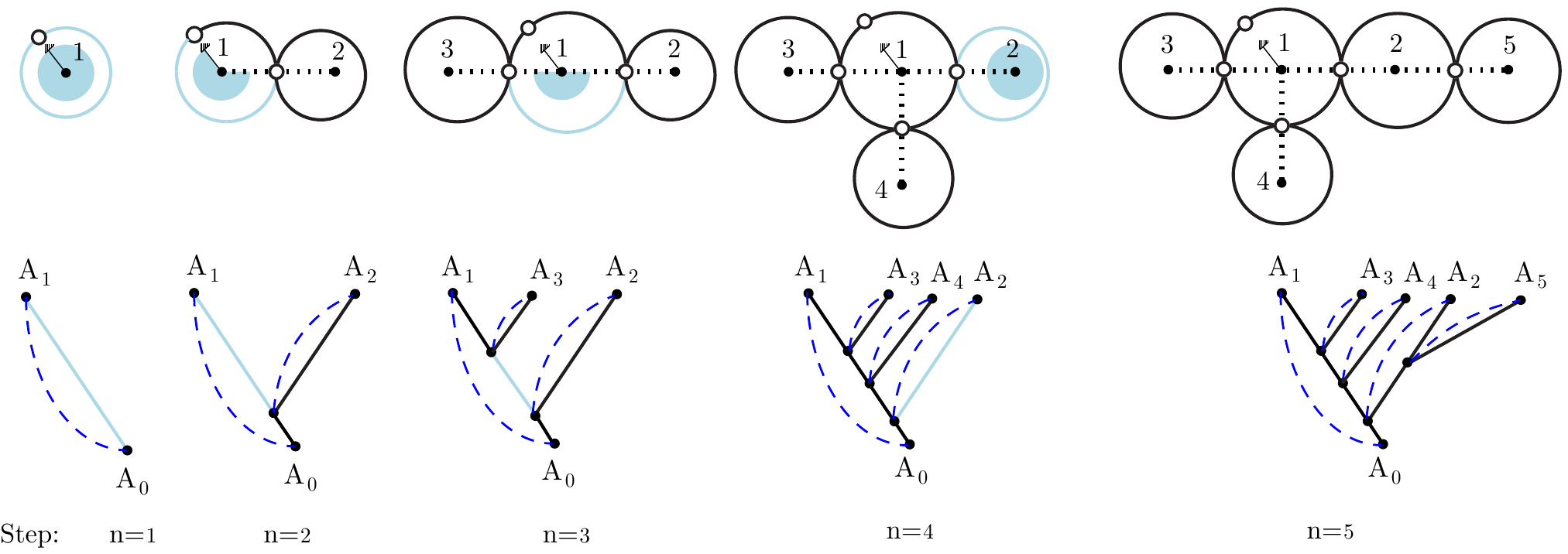}
 \caption{Illustration of the proof of Proposition~\ref{prop:coupling}, where in addition the edges that are split have been coupled in such a way that $\mathsf{Loop}( T^\multimap_{n})=\mathsf{Glu}( {\mathbf{B}}_{n})$. The first line represents the evolution of $(T^\multimap_{n}, \mathsf{Loop}( T^\multimap_{n}) )_{1 \leq n \leq 5}$ and the second the evolution of $({\mathbf{B}}_{n},\mathsf{Glu}( {\mathbf{B}}_{n}))_{1 \leq n \leq 5}$. In the second line, the dashed lines represent the identifications that are made.}
 \label{fig:coupling}
 \end{center}
 \end{figure}
\end{proof} 

\subsection{Definition of the Brownian looptree} \label {sec:br}

In this section, we define the Brownian looptree $ \mathcal{L}$. We first introduce some notation concerning continuous trees. 
A metric space $ ( \mathcal{T},d)$ is an $ \R$-tree if it contains no cycle and if for any points $x,y \in \mathcal{T}$ there exists a unique geodesic $ \llbracket x,y \rrbracket$ between $x$ and $y$ which isometric to a segment of $ \mathbb{R}$  (see \cite[Sec.~3]{LG05} for a more detailed definition). Moreover we impose that $ ( \mathcal{T},d)$  be compact. 
To mark the difference between $\R$-trees and the regular trees used up to now, we will sometimes call the latter discrete trees. 

We extend the notation introduced for discrete trees in Sec.~\ref {sec:dt} to continuous trees.  
If $ \mathcal{T}$ is an $ \R$-tree and  $x_{0},x_{1}, \ldots, x_{k} \in \mathcal{T} $ are distinct points, we let
$$ \mathsf{Span}(  \mathcal{T} ; x_{0}, \ldots,x_{k})= \bigcup_{0 \leq  i,j \leq k} \llbracket x_{i}, x_{j} \rrbracket$$
be the $ \R$-tree spanned by these vertices. The degree of a point $x \in \mathcal{T}$ is the number of connected components of $\mathcal{T} \backslash \{x\}$. A leaf is a point of degree $1$. A labeled $ \R$-tree is a pair consisting of an $ \R$-tree $ \mathcal{T} $ and a (finite or infinite) collection of leaves of $ \mathcal{T}$.  

Consider a labeled compact $ \R$-tree $\boldsymbol{\mathcal{T}} = ( \mathcal{T} ; (x_{i})_{0 \leq i < N})$, where $N \in \overline{\N}_{+} := \{1,2, \ldots\} \cup  \{+ \infty\}$, and assume that $ \mathcal{T}$ is binary (this assumption is not necessary, but it holds in our case and simplifies the exposition). The gluing $ \mathsf{Glu}( \boldsymbol{  \mathcal{T} })$ of $  \boldsymbol{\mathcal{T}}$ is  the quotient compact metric space constructed as follows. Set $p_{1}=x_{0}$, and for $2 \leq i < N$, let $p_{i}$ be the point of $ \mathsf{Span}( \mathcal{T} ; x_{0},x_{1}, \ldots x_{i-1})$ which is the closest to $x_{i}$ in $\tau$. Write $\sim$ for the equivalence relation on $ \mathcal{T}$ generated by $p_{i} \sim x_{i}$ for $1 \leq i < N$. If $d$ denotes the graph distance on $ \mathcal{T}$, we define a pseudo-distance $\Delta$ on $\mathcal{T}$  by
$$ \Delta(a,b) = \inf \left\{\sum_{i=0}^{k} d(p_{i},q_{i}) : p_{0}=a; q_{k}=b \right\},$$ 
where the infimum runs over all choices of $k \in \N$ and points $(p_{i})_{0 \leq i \leq k}$ and $(q_{i})_{0 \leq i \leq k}$ so that $q_{i} \sim p_{i+1}$ for $0 \leq i \leq k-1$. 

In the case of an generic metric space, defining a ``gluing'' metric could yield to more identifications 
that those prescribed by $ \sim$. This is not the case in our setup, as explained next. 

\begin{lemma}\label{lem:delta}
	For every $a,b \in \mathcal{T}$, $ \Delta(a,b) = 0$ if and only if $a \sim b$.
\end{lemma}

\begin{proof}  We first check that $\sim$ is closed. For this, consider a sequence $(a_{i},b_{i})_{i \geq 0}$ converging to $(a,b)$ as $i \to\infty$. We can suppose without loss of generality that  all the points $\{a_{i},b_{i}\}_{i \geq 0}$ are distinct. Since $ \mathcal{T}$ is binary, it is simple to see that we have $ u \sim v$ if and only if $u=v$ or $ \{ u,v\}= \{ p_{i},x_{i}\} $ for a certain $i \geq 1$. In particular, if $\{a,b\} \ne \{c,d\}$ with $a\sim b$ and $c \sim d$ we must have $\rrbracket a,b\llbracket \, \cap \,  \rrbracket c,d \llbracket = \varnothing$. By compactness this implies that $ d(a_{i},b_{i}) \to 0$, and hence $a=b$. The relation $\sim$ is thus closed.

Now let $a,b \in \mathcal{T}$ be such that $ \Delta(a,b) =0$.  For every $i \geq 1$, we denote by $a_{i}$ and $b_{i}$ the projections (i.e. closest point) of respectively $a$ and $b$ on $ \mathsf{Span}( \mathcal{T} ; x_{0}, \ldots , x_{i})$. A moment's thought shows that $a_{i}$ must by equal to $b_{i}$ inside $ \mathsf{Glu}(\mathsf{Span}( \mathcal{T} ; x_{0}, \ldots , x_{i}) ; x_{0}, \ldots , x_{i})$, since otherwise we would have $ \Delta(a,b)>0$. In particular, $a_{i} \sim b_{i}$. As $i \to \infty$, we have $a_{i} \to \tilde{a}$ where $ \tilde{a}$ is the projection of $a$ on the closure of 
$\bigcup_{i} \mathsf{Span}( \mathcal{T} ; x_{0}, \ldots , x_{i})$, and similarly $b_{i} \to \tilde{b}$. If $a \ne \tilde{a}$ (or $ b \ne \tilde{b}$), we would have $ \Delta(a,b)>0$ since there would exist a small ball around $a$ unaffected by the gluings. Hence $(a,b)= (\tilde {a}, \tilde {b})$, and $a_{i} \to a$ and $b_{i} \to b$. Since $\sim$ is closed, we have $a \sim b$ as desired. 
\end {proof} 

Using the above we may deduce (see for instance \cite[Exercise 3.1.14]{BBI01}) that  
\begin{equation}\label{eq:defrabbit}
	\mathsf{Glu}( \mathcal{T} ; (x_{i})_{0 \leq  i < N}) \quad := \quad (\mathcal{T} / \sim, \Delta).
\end{equation}
is a \emph{compact} metric space, which we call the (continuous) gluing of $( \mathcal{T} ; (x_{i})_{0 \leq  i < N})$. 
We shall denote by $ \pi : \mathcal{T} \rightarrow \mathsf{Glu}( \mathcal{T} ; (x_{i})_{0 \leq  i < N}) $ the canonical projection.

In the case where $ \mathcal{T} = \mathcal{T}_{ \mathbf{e}}$ is  the Brownian CRT 
and $x_{i} = X_{i}$ for $i \geq 0$ is a sequence of i.i.d.~random variables sampled according to the mass measure of $ \mathcal{T}_{ \mathbf{e}}$, the random compact metric space $ \mathcal{L}=\mathsf{Glu}( \mathcal{T}_{ \mathbf{e}}, (X_{i})_{i \geq 0})$ is called the Brownian looptree.

\begin{remark}
The Brownian looptree may also be constructed through a line breaking procedure, 
very similar to the one designed by Aldous to construct the Brownian CRT (see \cite[Theorem 7.6]{Pit06}). 
Consider $0 <\theta_{1}< \theta_{2} < \cdots$ to be the points of a Poisson point process on $\R_{+}$ with intensity $t/ 2 \cdot dt$.
Break the line $[0, \infty)$ at points $\theta_{k}$ to create segments of length $\theta_1, \theta_2 - \theta_1, \dots$. 
Glue the two end-points of each such segment together to create metric circles $ \mathcal{C}_{1}, \mathcal{C}_{2}, \dots$. 
Construct recursively metric spaces $ \mathcal{G}_{1}, \mathcal{G}_{2}, \dots$ 
by setting $ \mathcal{G}_{1} = \mathcal{C}_{1}$ and, for each $k \geq 1$, 
glueing $ \mathcal{C}_{k+1}$ to a point chosen uniformly at random on $ \mathcal{G}_{k}$.
The Brownian looptree is then the completion of $ \cup_{k \geq 1} \mathcal{G}_{k}$.
\end{remark}

\subsection{Convergence towards the Brownian looptree}\label{sec:proof}\label{sec:GH}

We briefly describe the $k$-pointed Gromov--Hausdorff topology (we refer to \cite{BBI01, Eva08,Mie09} for additional details). 
A $k$-pointed compact metric space is a triple $(E,d, (x_{1}, \ldots,x_{k}))$, where $(E,d)$ is a compact metric space and $x_{1}, \ldots,x_{k} \in E$. Two $k$-pointed compact metric spaces are said to be isometric if there exists an isometry between them mapping the $k$ distinguished points of one of them to the distinguished points of the other (preserving the order). 
The set of isometry classes of $k$-pointed compact metric spaces is endowed with the $k$-pointed Gromov--Hausdorff distance defined next. If $(E,d,( x_{1}, \ldots,x_{k}))$ and $(E',d',(x'_{1}, \ldots,x'_{k}))$ are two $k$-pointed compact  metric spaces, 
\begin{eqnarray*} 
\op{\mathrm{d_{GH}}}\big((E,d, (x_{1}, \ldots,x_{k})) , (E',d',( x'_{1}, \ldots,x'_{k}))\big) 
&=& \inf \left\{\op{d}_{\op{H}}^F(\phi(E),\phi'(E'))  \vee \max_{1 \leq i \leq k} \delta( \phi(x_{i}), \phi'(x'_{i}))\right\}, 
\end{eqnarray*} 
where the infimum is taken over all choices of metric spaces $(F,\delta)$ and  isometric embeddings $\phi : E \to F$ and $\phi'  : E' \to F$ of $E$ and $E'$ into $F$, and where $ \mathrm{d}_{ \mathrm{H}}^F$ denotes the Hausdorff distance between compacts sets in $F$. 
The $k$-pointed Gromov--Hausdorff distance is indeed a metric on the space of isometry classes of $k$-pointed compact metric spaces. It renders this space separable and complete. For $k=0$, $\op{\mathrm{d_{GH}}}$ is the usual Gromov--Hausdorff distance on (isometry classes of) compact metric spaces.   

\medskip

We now state a continuity proposition inspired from \cite[Proposition 12]{CH13}. 
If $(E,d)$ is a metric space and $x_{0}, \ldots, x_{n} \in E$, we say that $x_{0}, \ldots, x_{n}$ is an $ \varepsilon$-net in $E$ if $E= \displaystyle \bigcup_{0 \leq i \leq n}  \{y \in E; d(y,x_{i})< \varepsilon\}$.

\begin{proposition} \label{prop:lipschitz} 
Let $(\tau ; (x_{i})_{0 \leq i <N})$ be either a labeled discrete tree, or a labeled $ \R$-tree, with $N \in \overline{\N}_{+}$. Then, for every integer $0 \leq k < N$,\begin{eqnarray*}    \mathrm{d_{GH}} \big( \mathsf{Glu}(\tau ; (x_{i})_{0 \leq i \leq k}) , \mathsf{Glu}( \tau ; (x_{i})_{0 \leq i <N})\big) \leq  2 \inf\left\{\varepsilon>0 ;  x_{0}, \ldots, x_{k} \mbox{ is an } \varepsilon \mbox{-net in } \tau\right\},\end{eqnarray*}  where $\tau$ is equipped with its graph distance in the discrete case or with its metric in the continuous case.
\end{proposition}

\proof  For $k \geq 0$, set $\tau_{k} = \mathsf{Span}( \tau ; x_{0}, \ldots , x_{k})$. We clearly have
  \begin{eqnarray} \label{eq:verysimple} \mathrm{d_H}(\tau, \tau_{k}) \leq \inf\left\{\varepsilon>0 ;  x_{0}, \ldots, x_{k} \mbox{ is an } \varepsilon \mbox{-net in } \tau\right\}. \end{eqnarray}  We thus can bound  $\mathrm{d_{GH}} \big( \mathsf{Glu}(\tau ; (x_{i})_{0 \leq i \leq k}) , \mathsf{Glu}( \tau ; (x_{i})_{0 \leq i <N})\big)$ above by 
    \begin{eqnarray*} && \mathrm{d_{GH}} \big( \mathsf{Glu}(\tau ; (x_{i})_{0 \leq i \leq k}) , \mathsf{Glu}( \tau_{k} ; (x_{i})_{0 \leq i \leq k})\big)
    +  \mathrm{d_{GH}} \big( \mathsf{Glu}( \tau_{k} ; (x_{i})_{0 \leq i  \leq k}) , \mathsf{Glu}( \tau ; (x_{i})_{0 \leq i <N})\big),  \end{eqnarray*} which is less than or equal to $\mathrm{\mathrm{d_H}}( \tau, \tau_{k}) +   \mathrm{\mathrm{d_H}}( \tau, \tau_{k})$ since
 $\mathsf{Glu}( \tau_{k} ; x_{0}, \ldots, x_{k}) \to \mathsf{Glu}( \tau ; x_{0}, \ldots , x_{k})$ and similarly $ \mathsf{Glu}( \tau_{k} ; x_{0}, \ldots , x_{k}) \to \mathsf{Glu}( \tau ; (x_{i})_{0 \leq i < N})$ are isometric embeddings and $\mathsf{Glu}$ is a contraction. Combining this with~\eqref{eq:verysimple} finishes the proof. \endproof

Before proceeding to the proof of of Theorem~\ref{thm:scaling}, we state a final simple property that we will not prove. 
\begin{lemma}\label{lem:gluecv}
	Fix an integer $k \geq 0$. 
	Let $(\boldsymbol{\mathcal{T}}^{(n)})_{n \geq 1}=\big( {\mathcal{T}}^{(n)}; x_{0}^{(n)},  \ldots, x^{(n)}_{k}\big)_{n \geq 1}$ 
	be a sequence of labeled $ \R$-trees
	and $\boldsymbol{\mathcal{T}}=( {\mathcal{T}}; x_{0},  \ldots, x_{k})$ 
	be a labeled $ \R$-tree. 
	Suppose that $ \boldsymbol{\mathcal{T}}^{(n)} \rightarrow \boldsymbol{\mathcal{T}}$ 
	holds almost surely for the $k$-pointed Gromov--Hausdorff topology. 
	Then  $ \rab(\boldsymbol{\mathcal{T}}^{(n)}) \rightarrow  \rab(\boldsymbol{\mathcal{T}})$ also holds almost surely 
	for the Gromov--Hausdorff topology. 
\end{lemma}

\proof[Proof of Theorem~\ref{thm:scaling}]
 Recall from Section~\ref{sec:coupling} the notation $ \mathbf{B}_{n}=(B_{n}; A_{0}, \ldots , A_{n})$ for the sequence of trees grown by Rémy's algorithm.  By \cite[Theorem 5 (ii)]{CH13}, there exists a pair $(\mathcal{T}_{ \mathbf{e}}, (X_{i}; i \geq 0))$, where $\mathcal{T}_{ \mathbf{e}}$ is a Brownian CRT and $(X_{i}; i \geq 0)$ is a collection of i.i.d.~vertices sampled according to its mass measure, such that for every $k \geq 1$ we have the following convergence for the $k+1$-pointed Gromov-Hausdorff topology
\begin{eqnarray} \label{eq:cvpsremy} 
\big( n^{-1/2}\cdot B_{n} ; A_{0}, \ldots , A_{k}\big) 
& \xrightarrow[n\to\infty]{a.s.} & 
\big( 2 \sqrt{2} \cdot \mathcal{T}_{ \mathbf{e}}; X_{0}, \ldots , X_{k}\big).
\end{eqnarray}
Hence, by Lemma~\ref{lem:gluecv}, the following holds in the regular Gromov-Hausdorff topology
\begin{eqnarray} \label{cv2}   
	n^{-1/2} \cdot \mathsf{Glu}(B_{n} ; A_{0}, \ldots, A_{k}) 
	& \xrightarrow[n\to\infty]{a.s.} & 
	2 \sqrt{2} \cdot \mathsf{Glu}( \mathcal{T}_{ \mathbf{e}} ; X_{0}, \ldots , X_{k}).  
\end{eqnarray}
For $0 \leq k\leq n$, set $ \mathsf{L}_{n}^{(k)} =  \mathsf{Glu}(B_{n}; A_{0}, \ldots , A_{k})$ 
so that $ \mathsf{Glu}( \mathbf{B}_{n}) = \mathsf{L}_{n}^{(n)}$. 
Also set $ \mathcal{L}_{k} = \mathsf{Glu}( \mathcal{T}_{ \mathbf{e}} ; X_{0}, \ldots , X_{k})$. 
Now, for $n\geq k \geq 1$, 
\begin{eqnarray*}
	\dGH\left( \frac{\mathsf{L}_{n}^{(n)}}{ \sqrt{n}}, 2 \sqrt {2} \cdot \mathcal{L}\right) 
	&\leq&  
	d_{ \mathrm {H}}\left( \frac{\mathsf{L}_{n}^{(n)}}{ \sqrt{n}},  \frac{\mathsf{L}_{n}^{(k)}}{ \sqrt{n}}\right)  
	+ \dGH\left( \frac{\mathsf{L}_{n}^{(k)}}{ \sqrt{n}}, 2 \sqrt{2} \cdot \mathcal{L}_{k}\right) 
	+ d_{ \mathrm {H}}(2 \sqrt{2} \cdot \mathcal{L}_{k}, 2 \sqrt{2}\cdot \mathcal{L}),    
\end{eqnarray*}
Denote by respectively $U_{n,k}, V_{n,k}$ and $W_{k}$ the three terms appearing in the previous sum. 
In order to prove that the right-hand side above converges to $0$ as $n \to \infty$, 
we will first take the $\limsup$ of the above as $n \to \infty$, then make $k$ tend to $\infty$. 

By~\eqref{cv2}, $ \lim_{n \rightarrow \infty} V_{n,k}=0$ for every fixed $k \geq 1$. 
Also, since $ (X_{i})_{i \geq 0}$ is a.s.\,dense in $ \mathcal{T}_{ \mathbf{e}}$, 
by Proposition~\ref{prop:lipschitz} we get that $ \lim_{k \rightarrow \infty} W_{k}=0$. 
By Proposition~\ref {prop:lipschitz}, 
$$
	U_{n,k} = 
	d_{ \mathrm {H}}\left( \frac{\mathsf{L}_{n}^{(n)}}{ \sqrt{n}}, \frac{\mathsf{L}_{n}^{(k)}}{ \sqrt{n}} \right) 
	\leq  2 \inf\left\{\varepsilon>0 ;  A_{0}, \ldots, A_{k} \mbox{ is an } \varepsilon \mbox{-net in }  n^{-1/2} \cdot B_{n}\right\}.
$$
But by~\eqref{eq:cvpsremy} we have 
$$ 
	\inf\left\{\varepsilon>0 ;  A_{0}, \ldots, A_{k} \mbox{ is an } \varepsilon \mbox{-net in }  n^{-1/2} \cdot B_{n}\right\} 
	\xrightarrow[n \rightarrow \infty]{a.s.}  
	\inf\left\{\varepsilon>0 ;  X_{0}, \ldots, X_{k} \mbox{ is an } \varepsilon \mbox{-net in }  2 \sqrt {2} \cdot \mathcal{T} _{\mathbf{e} } \right\}.
$$
We deduce that 
$$ 
	\limsup_{k \to \infty} \limsup_{n \to \infty} 
	d_{ \mathrm {H}}\left( \frac{\mathcal{L}_{n}^{(n)}}{ \sqrt{n}},  \frac{\mathcal{L}_{n}^{(k)}}{ \sqrt{n}}\right) 
	\leq 4 \sqrt{2} \cdot \limsup_{k \to \infty}\inf\left\{\varepsilon>0 ;  X_{0}, \ldots, X_{k} \mbox{ is an } \varepsilon \mbox{-net in }  \mathcal{T} _{\mathbf{e} } \right\}=0,
$$
since the collection $(X_{i}; i \geq 0)$ is almost surely dense in $ \mathcal{T}_{ \mathbf{e}}$.
The proof is complete.
\endproof

\subsection{Convergence towards Brownian looptrees for general seeds}
\label{sec:genseed}

In this section we prove the Corollary \ref{cor:scaling}.
In order to describe the construction of $\mathcal{L}^{(S)}$ and prove this result, a preliminary discussion is required 
on how $ T_{n}^{(S)}$ may be constructed from independent copies of the processes $(T_{n}^{\multimap})_{n \geq 1}$. 

For $n \geq 0$ and $N \geq 1$, denote by $ \pol(n,N)$ the law after $n$ draws of the state of a P\'olya urn with $N$ colors,
starting with one ball of each color and diagonal replacement matrix $\textrm{Diag}(2,2, \ldots,2)$.
In other words, consider an urn with $N$ balls of different colors. 
At each step a ball is taken out uniformly at random, inspected, 
and then put back in the urn along with two additional balls of the same color. 
Then $\pol(n,N)$ is the law of $(X_1^n, \dots, X_N^n)$, where $X_{i}^n$ represents the numbers of balls of the $i$-th color after $n$ draws.  
If $ \underline{P}$ is a plane planted tree and $c$ is a corner of some plane tree $S$, 
then a new plane tree may be obtained by gluing $\underline{P}$ inside $c$, as depicted in Fig.~\ref{fig:glue}.

\begin{figure}[!h]
  \begin{center}
    \includegraphics[width=0.6 \linewidth]{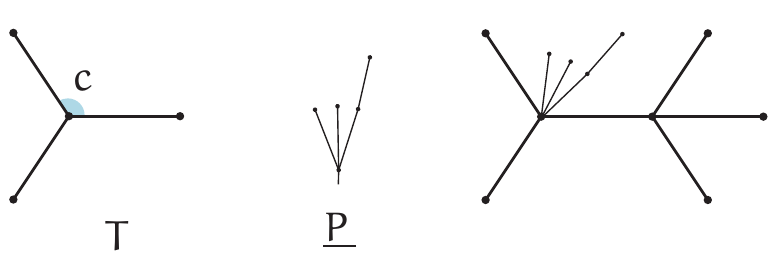}
    \caption{A plane tree $T$ with a distinguished corner $c$, 
      a plane planted tree, and the plane tree obtained by gluing $\underline{P}$ inside $c$.}
    \label{fig:glue}
  \end{center}
\end{figure}
 
\begin{proposition}\label{prop:decomp} 
  Fix a plane tree $S$ and let $(c_{1}, \ldots,c_{N})$ be an exhaustive enumeration of its corners with $N:=2|S|-2$. 
  If $n \geq |S|$ is an integer,  let $(\alpha^n_1, \dots, \alpha^n_N)$ be a random variable sampled according to $ \pol(n - |S|,N)$.  
  Then, conditionally on $(\alpha^n_1, \dots, \alpha^n_N)$, let 
  $\underline{P}^n_1, \dots, \underline{P}^n_N$ be independent random variables distributed as respectively 
  $T_{(\alpha^n_1+1)/{2}}^{\multimap}, \dots, T_{(\alpha^n_N+1)/{2}}^{\multimap}$.  
  Finally, let $ S_{n}$ be the tree obtained by gluing, for every $1 \leq i \leq N$, the planted tree $ \underline{P}^n_{i}$ 
  in each each corner $c_i$ of $S$. Then $ S_{n}$ has the same law as $T_n^{(S)}$.
\end{proposition}

Rather than a formal proof, we give a brief explanation of this fact. 
Combined with Figure~\ref{fig:glue3}, it should be enough to convince the reader.
As $T_n^{(S)}$ grows from $S$, vertices are added sequentially. 
For every $1 \leq i \leq N$, there are $(\alpha^{n}_i - 1)/{2}$ vertices added  to the corner $c_i$ of $S$  (that is either direct neighbours of $c_i$, or linked to $c_i$ by edges not belonging $S$). In particular, the subtree of $T_n^{(S)}$ emanating from $c_i$ is a planted tree with $\alpha^{n}_i$ corners 
(including the corners at its base). 
Thus, in order to construct $T_{n+1}^{(S)}$ from $T_n^{(S)}$, in order to construct $T_{n+1}^{(S)}$, 
the new vertex is added in the tree emanating from $c_i$ with probability $\alpha^{n}_i/ \sum_ {j=1} ^{N}\alpha^{n}_j$. 
This shows that $(\alpha^{n}_1, \dots, \alpha^{n}_N)$ indeed follows the law $\pol(n,N)$.
Moreover, conditionally on the number $\frac{\alpha^{n}_i - 1}{2}$ of vertices added to $c_i$, 
these vertices are added following the rules of the LPAM starting with $ \multimap$ as the seed. 
Hence the tree emanating from $c_i$ in $T_n^{(S)}$ has the law of $T_{\frac{\alpha^{n}_i+1}{2}}^{\multimap}$. 
Finally, the trees growing inside the different corners of $T$ are independent conditionally on their size.

\begin{figure}[!h]
  \begin{center}
    \includegraphics[width=0.8 \linewidth]{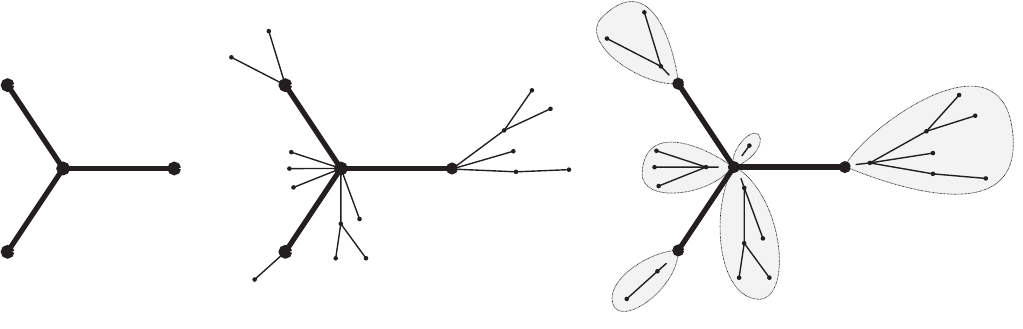}
    \caption{The tree $T^{(S)}_{n}$ is obtained by gluing planted trees in the corners of the seed $S$.}
    \label{fig:glue3}
  \end{center}
\end{figure}

We are now ready to describe the construction of the limit space $\mathcal{L}^{(S)}$ of Corollary \ref{cor:scaling}.
For this we need to introduce notation.

If $ \underline{P}$ is a planted tree, define a modified looptree $ \widetilde{\mathsf{Loop}}(\underline{P})$ 
by "cutting" ${\mathsf{Loop}}(\underline{P})$ at the vertex associated with the root half-edge of $\underline{P}$.
More precisely delete this vertex and add two distinct vertices as endpoints 
of the two edges of ${\mathsf{Loop}}(\underline{P})$ incident to the removed vertex.
Let  $g(\underline{P})$, resp.~$d(\underline{P})$, denote the endpoints of the edge to the left, resp. right, of the 
root half-edge of $\underline{P}$, when the latter is oriented towards its only endpoint. 
 \begin{figure}[!h]
 \begin{center}
 \includegraphics[width=0.5 \linewidth]{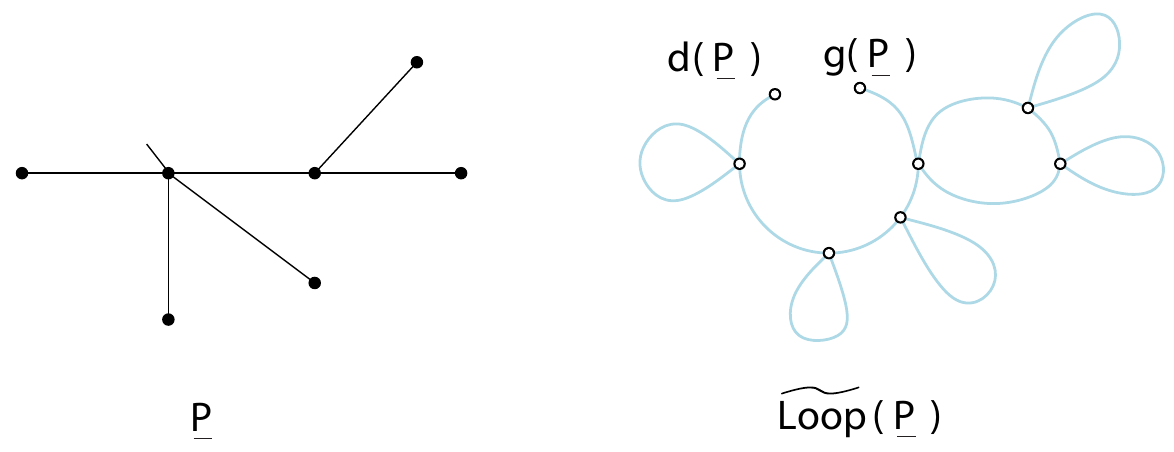}
 \caption{A planted tree $ \underline{P}$ and its associated modified looptree $\widetilde{\mathsf{Loop}}(\underline{P})$.}
 \label{fig:looptilde}
 \end{center}
\end{figure}

 A simple extension of Theorem~\ref{thm:scaling} then shows that we have the following almost sure $2$-pointed Gromov--Hausdorff convergence
\begin{equation*}
	\left( n^{-1/2} \cdot \widetilde{\mathsf {Loop}}(T_{n}^\multimap) ; \ g(T_{n}^\multimap),d(T_{n}^\multimap) \right)  
	\quad\mathop{\longrightarrow}^{a.s.}_{n \rightarrow \infty} \quad  
	\left(  2 \sqrt{2} \cdot  \widetilde{\mathcal{L}}\  ; \ \pi(X_{0}), \pi(X_{1}) \right),
\end{equation*}
where $ \widetilde{ \mathcal{L}}$ is constructed exactly as $ \mathcal{L}$ except that we do not make the identification $ X_{0} \sim X_{1}$. Equivalently, $ \widetilde{\mathcal{L}}$ is obtained from the  Brownian looptree by ``cutting" it at the vertex $ \pi( X_{0})$ and distinguishing the two newly obtained points. To simplify notation, write $\boldsymbol{\widetilde{  \mathcal{L}}}=(\widetilde{\mathcal{L}}\  ; \ \pi(X_{0}), \pi(X_{1}))$.
   
Assume that $T_n^{(S)}$ is constructed as in Proposition~\ref{prop:decomp}.
By standard results concerning P\'olya urns (see e.g 	\cite{Ath69} or \cite[Prop.~3]{CMP14}), we have
$$ 
	\left( \frac{\alpha^n_1}{n}, \dots, \frac{\alpha^n_N}{n} \right)  
	\quad\mathop{\longrightarrow}^{a.s.}_{n \rightarrow \infty} \quad 
	2 \cdot (\alpha_{1},\ldots, \alpha_{N}),
$$
where $(\alpha_{1},\ldots, \alpha_{N})$ follows the Dirichlet distribution $\mathrm{Dir} \left( \frac{1}{2},\frac{1}{2}, \ldots,\frac{1}{2} \right)$. It follows there exists a collection $({\widetilde{\mathcal{L}}}_{i}\  ; \ \pi(X^0_{i}), \pi(X^1_{i}))_{ 1 \leq i \leq N}$ of independent pointed modified Brownian looptrees such that the convergence
\begin{equation}\label{eq:tilde1} 
	\left( \frac{\widetilde{\mathsf {Loop}}(\underline{P}^n_i)}{ \sqrt{n} }\  ; \ g(\underline{P}^n_i),d(\underline{P}^n_i) \right)  
	\quad\mathop{\longrightarrow}^{a.s.}_{n \rightarrow \infty} \quad  
	\left(  2 \sqrt{2} \cdot \sqrt{ \alpha_{i}} \cdot {\widetilde{\mathcal{L}}}_{i}\  ; \ \pi(X^0_{i}), \pi(X^1_{i}) \right)
\end{equation}
holds in the $2$-pointed Gromov--Hausdorff topology for every $1 \leq i \leq N$. 

Then $\mathcal{L}^{(S)}$ is obtained by gluing these metric spaces along the structure given by the seed,
as described next. 
Let $\{m(e): e \in E(S)\}$ denote the collection of midpoints of edges of $S$. 
For each corner $c_i$ of $S$, let $e_g(i)$ and $e_d(i)$ be the edges to its left and right, respectively 
(note that they are not necessarily distinct). 
For every $i \geq 1$, identify the points $\pi(X^0_{i})$, $\pi(X^1_{i})$ of $2 \sqrt{2} \cdot \sqrt{ \alpha_{i}} \cdot {\widetilde{\mathcal{L}}}_{i}$ to $m(e_g(i))$ and $m(e_d(i))$, respectively. 
This creates a compact metric space
which we denote by $\mathcal{L}^{(S)}$.
The same construction may be performed in the discrete setting, see Fig.~\ref{fig:gluesurseed} for an illustration.

Corollary~\ref{cor:scaling} follows readily from~\eqref{eq:tilde1}
and from the fact that $\mathsf{Loop}(  T_{n}^{(S)})$ 
may be obtained from the modified looptrees of $\underline{P}^n_1, \dots, \underline{P}^n_N$
in the same way as $\mathcal{L}^{(S)}$ 
is obtained from $ \boldsymbol{\widetilde{\mathcal{L}}}_{1}, \dots,  \boldsymbol{\widetilde{\mathcal{L}}}_{N}$.

\begin{figure}[!h]
 \begin{center}
 \includegraphics[width=0.8 \linewidth]{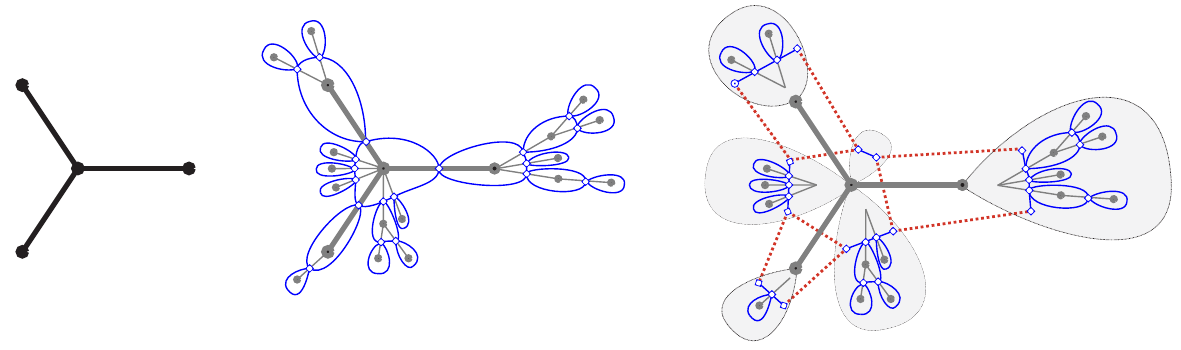}
 \caption{The discrete looptree $ \mathsf{Loop}(  T_{n}^{(S)})$ can be constructed from the collection of modified looptrees $(\widetilde{\mathsf {Loop}}(\underline{P}^n_i))_{i}$ by identifying vertices connected by dashed edges.}
 \label{fig:gluesurseed}
 \end{center}
\end{figure}

\subsection{Dimension of the Brownian looptree}
 In this section we establish Proposition~\ref{prop:hausdorff}.
 
Write $ \mathcal{L} = \mathsf{Glu}( \mathcal{T}_{ \mathbf{e}}, (X_{i})_{i \geq 0})$, where $ \mathcal{T}_{ \mathbf{e}}$ is a Brownian CRT and $(X_{i})_{i \geq 0}$ is a collection of independent leaves sampled according to its mass measure $ \mu$. Recall that $\pi : \mathcal{T}_{ \mathbf{e}} \to \mathcal{L}$ is the canonical projection. The upper bound on the Hausdorff dimension is a consequence of the fact that $ \pi$ is a contraction. Since $ \mathrm{dim}( \mathcal{T}_{ \mathbf{e}}) =2$, it follows that $ \mathrm{dim}( \mathcal{L}) \leq 2$ (see e.g. \cite[Theorem 7.5]{Mat95}). To establish the lower bound, we will use the probability mass measure $ \nu$ on $ \mathcal{L}$, which is defined as push-forward of $ \mu$ by the canonical projection. We shall show the following result: 

 \begin{lemma} \label{lem:borneinfDH}For every $ \varepsilon >0$, almost surely, for $\nu$-every $x$ we have 
 $$ \limsup_{r \to 0}\frac{\nu (B_{ r}(x))}{r^{2- \varepsilon}} =0,$$
 where $ B_{r}(x)$ denotes the open ball of radius $r$ around the point $x$ in $ \mathcal{L}$.
 \end{lemma}

 By standard density theorems for Hausdorff measures \cite[Theorem 8.8]{Mat95} (this reference covers the case of measures on $ \R^n$, but the proof remains valid here), this implies that the Hausdorff dimension of $ \mathcal{L}$ is greater than or equal to $ 2-\varepsilon$, almost surely. The lower bound will thus follow. 
 
 The rest of this section is  devoted to the proof Lemma~\ref{lem:borneinfDH}. To simplify, we say that a point is chosen uniformly in $ \mathcal{T}_{ \mathbf{e}}$ if it is sampled according to its mass measure $ \mu$. Consider an additional uniform random leaf  $Y \in \mathcal{T}_{ \mathbf{e}}$, independent of $(X_{i})_{i \geq 0}$. Note that almost surely, $Y \neq X_{i}$ for every $i \geq 0$. 
 We shall prove that for every $ \epsilon \in (0,2)$, almost surely, 
  $$ \limsup_{r \to 0}\frac{\nu (B_{ r}( \pi(Y)))}{r^{2- \varepsilon}} =0,$$
  By Fubini's theorem, this indeed implies Lemma~\ref{lem:borneinfDH}. To this end, define a nested sequence of rooted subtrees $  \mathcal{T}_{ \mathbf{e}} = \mathcal{T}^{1} \supset \mathcal{T}^2 \supset \mathcal{T}^{3}  \supset \cdots$ all   containing the point $Y$ and defined recursively as follows.   First, set $\mathcal{T}_{ \mathbf{e}} = \mathcal{T}^{1}$ which is rooted at $\tilde{P}_{1} :=  X_{0}$. For every $j \geq 1$, if $\mathcal{T}^{1}, \ldots, \mathcal{T}^{j}$ have been constructed, set $ k_{j}= \min  \{i \geq 1; \ X_{i} \in \mathcal{T}^{j} \}$. Next, consider  $ \tilde{P}_{j+1}$ the branching point between $\tilde{P}_{j} , X_{k_{j}}$ and $Y$ (if $a,b,c$ are different leaves of $\mathcal{T}_{ \mathbf{e}}$, the branching point between $a,b$ and $c$ is defined to be the unique element of $\llbracket a ,b\rrbracket  \cap \llbracket b, c \rrbracket \cap \llbracket a,   c\rrbracket$). The tree $ \mathcal{T}^{j+1}$ is  finally defined to be the subtree of $ \mathcal{T}^j \backslash \{ \tilde{P}_{j+1}\}$ containing $Y$ to which we add the vertex $\tilde{P}_{j+1}$.
Moreover  $\tilde{P}_{j+1}$  is declared to be the root of  $ \mathcal{T}^{j+1}$. 
We refer to Fig.~\ref{fig:boucles} for an illustration.
 
  \begin{figure}[!h]
   \begin{center}
   \includegraphics[width=7.5cm]{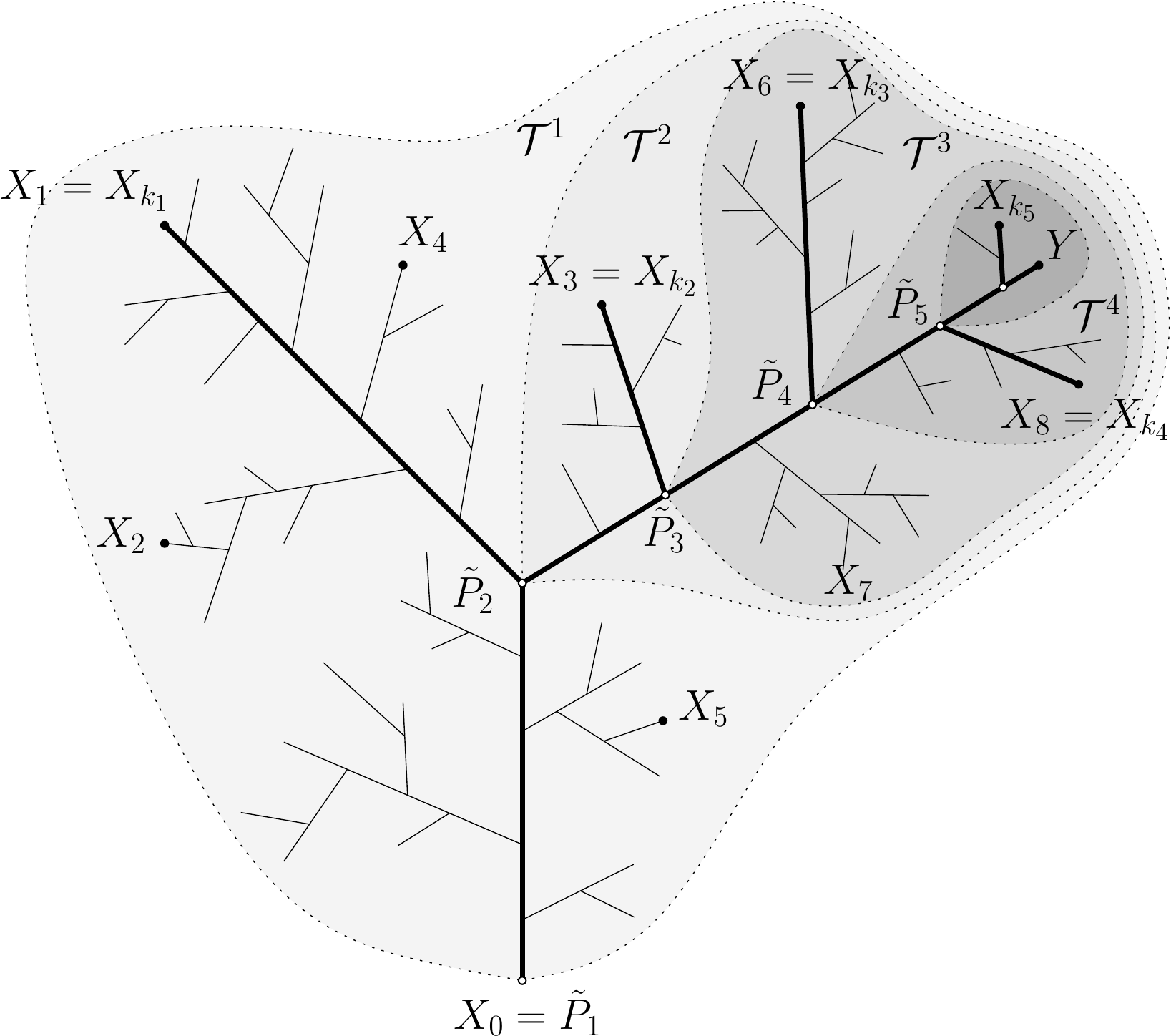} \hspace{1cm}
      \includegraphics[width=7.5cm]{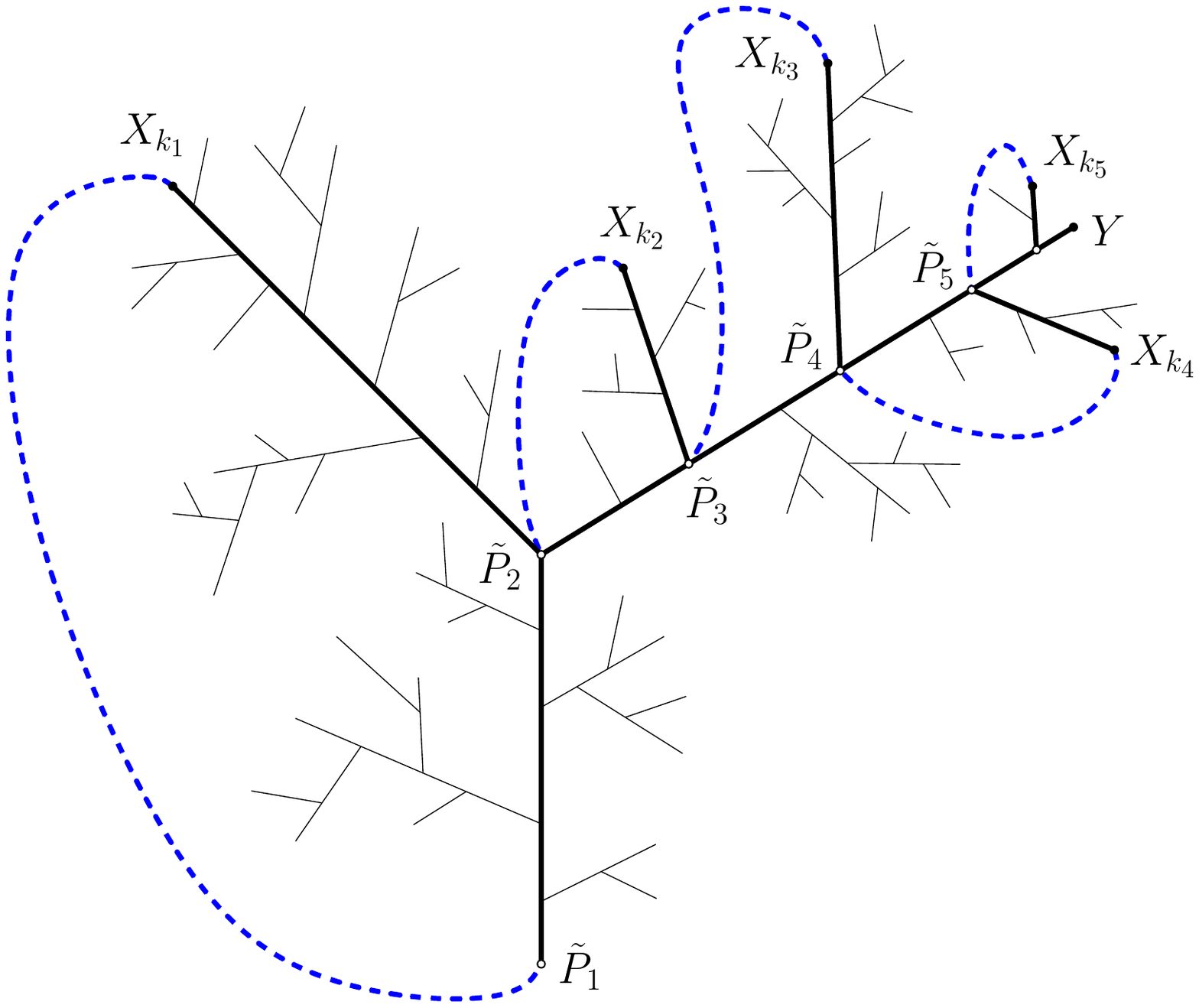}
   \caption{Illustration of the construction of $ \mathcal{T}^i$ for $i \geq 0$ and of the proof of Proposition~\ref{prop:bi}.}
   \label{fig:boucles}
   \end{center}
   \end{figure}

 \begin{proposition} \label{prop:arbreemboites} The following assertions hold.
 \begin{enumerate}
\item[(i)]  The process  $(-\log \mu ( \mathcal{T}^i))_{i \geq 1}$ is a random walk, 
	and its step distribution is an exponential random variable of parameter $1/2$.
\item[(ii)] For every $i \geq 1$,  the random tree $  \frac{1}{\sqrt{ \mu( \mathcal{T}^i})} \cdot \mathcal{T}^i$ 
has the law of a Brownian CRT. In addition $Y$ and $ \tilde{P}_{i}$ are two independent uniform leaves of $\mathcal{T}^i$.
\end{enumerate}
 \end{proposition}

 \proof	We prove the statement by induction on $i \geq 1$. For $i =1$, this is simply because $X_{0}$ and $Y$ are two independent leaves in $ \mathcal{T}_{ \mathbf{e}}$. By induction, at step $i \geq 1$, we assume that $ \mathcal{T}^{i}$ is a random multiple of a Brownian CRT and that $ Z_{0}:=\tilde{P}_{i}$ and $Z_{1}:=Y$ are two independent uniform leaves of $ \mathcal{T}^i$. Observe that  by construction, $Z_{2}:=X_{k_{i}}$ is a uniform leaf of $ \mathcal{T}^i$, independent of $(Z_{0},Z_{1})$. In addition, if $B$ denotes the branching point between $Z_{0},Z_{1}$ and $Z_{2}$ , note that $ \mathcal{T}^{i}$ is the union of three subtrees containing respectively $Z_{0}, Z_{1},Z_{2}$ and having $B$ as the only common element, and that the subtree $ \mathcal{T}^{i+1}$ is the one containing $Z_{1}$, rooted at $B$. It follows from Aldous' decomposition in three parts of the CRT  \cite[Theorem 2]{Ald94a} that $ \mu(\mathcal{T}^{i+1})=\alpha \cdot \mu(\mathcal{T}^i)$, where $\alpha$ is the first coordinate of a Dirichlet $\mathrm{Dir}(1/2,1/2,1/2)$ random variable independent of $ \mathcal{T}^i$, that $ \mathcal{T}^{i+1}$ has the same distribution as $\sqrt{ \alpha}$ times $ \mathcal{T}^i$,  and that $B$ and $Z_{1}$ are independent uniform leaves in $ \mathcal{T}^{i+1}$. This implies the second assertion. It is a simple matter to check that $ \alpha$ has density $(2 \sqrt{x})^{-1}$ on $[0,1]$, so that $- \log( \alpha)$ is distributed according to an exponential random variable of parameter $1/2$.  This completes the proof.
  \endproof
  
Now, for every tree $ \mathcal{T}^i$ we introduce the quantity $$ \mathcal{X}_{i} =  \min\Big( d_{ \mathcal{T}_{ \mathbf{e}}}( \tilde{P}_{i}, \tilde{P}_{i+1}), d_{ \mathcal{T}_{ \mathbf{e}}} ( X_{k_{i}}, \tilde{P}_{i+1}) \Big).$$
The reason for considering this random variable lies in the following geometric proposition:

 \begin{proposition} \label{prop:bi}For every $i \geq 2$ and every $x  \in \mathcal{T}_{ \mathbf{e}}$ such that $x \notin \mathcal{T}^{i}$, we have 
 $$  \Delta \big( \pi({x}), \pi({Y})\big) \geq  \mathcal{X}_{i}.$$
 \end{proposition}
 \proof[Proof (Sketch)] In the construction of the Brownian looptree from $ \mathcal{T}_{ \mathbf{e}}$ and $(X_{i})_{i\geq0}$, the points $ X_{k_{i}}$ are glued to $\tilde{P}_{i}$ for every $i \geq 0$.
 Specifically, each segment $\llbracket \tilde{P}_{i}, X_{k_{i}} \rrbracket$ becomes a loop denoted by $L_{i}$ in $ \mathcal{R}$, and the loops $L_{i}$ and  $L_{i+1}$ share the common point $ \pi(  \tilde{P}_{i+1}) = \pi( X_{k_{i}})$. It should then be clear that in $ \mathcal{R}$, the region $\pi ( \mathcal{T}^1 \backslash \mathcal{T}^{i})$ is separated from $\pi(Y)$ and that the only way to go from $  \pi(Y)$ to this region is to travel along $ \mathcal{R}$ and cross the loop $L_{i}$ from $ \pi(  \tilde{P}_{i+1})$ to $\pi(  \tilde{P}_{i})$ but this requires at least a length $ \mathcal{X}_{i}$. We leave the details to the reader. \endproof

\proof[Proof of Lemma~\ref{lem:borneinfDH}] By the first assertion of Proposition~\ref{prop:arbreemboites} and the strong law of large numbers we have  \begin{eqnarray} \label{eq:llnmasse} \frac{\log\big(\mu( \mathcal{T}^{i})\big)}{i}  & \xrightarrow[i\to\infty]{a.s.} -2.  \end{eqnarray} By the second assertion of the last proposition we have $  \mathcal{X}_{i}/\sqrt{ \mu ( \mathcal{T}^i)}^{} =  \mathcal{X}_{1}$ in distribution. In addition, by \cite[Theorem 2.11]{LG05}, $ \mathcal{X}_{1}$ has the same law as $\min(\ell_{1},\ell_{2})$ where $(\ell_{1},\ell_{2},\ell_{3})$ has density on $ \mathbb{R}_{+}^3$ given by $16 (\ell_{1}+\ell_{2}+\ell_{3}) e ^{-2 (\ell_{1}+\ell_{2}+\ell_{3})^2}\,\mathrm{d}\ell_{1}\mathrm{d}\ell_{2}\mathrm{d}\ell_{3}$.  From this expression, it is a simple matter to establish the existence of a constant $c>0$ such that for every $i \geq 1$, $ \P( \mathcal{X}_1 \geq i \mbox{ or } \mathcal{X}_1 \leq i^{-2}) \leq c i^{-2}$. An application of Borel--Cantelli's yields that almost surely, for every $i$ sufficiently large, $ \mathcal{X}_{i}/\sqrt{ \mu ( \mathcal{T}^i)}^{} \geq i^{-2}$ and $ \mathcal{X}_{i}/\sqrt{ \mu ( \mathcal{T}^i)}^{} \leq i^2$. Combining this with~\eqref{eq:llnmasse}, we get that 
 \begin{eqnarray} \label{eq:llndistance} \frac{\log( \mathcal{X}_{i})}{i}  & \xrightarrow[i\to\infty]{a.s.} -1.  \end{eqnarray}
 Now, by Proposition~\ref{prop:bi}, we have  $ B_{ \mathcal{X}_{i}}( \pi(Y))) \subset \pi(\mathcal{T}^{i})$. Noting that $ \mu( \pi^{-1}( \pi(A)))= \mu(A)$ for every $A \subset \mathcal{T}_{ \mathbf{e}}$, this implies that  $\nu( B_{\mathcal{X}_{i}}( \pi(Y))) \leq \mu( \mathcal{T}^{i})$. By combining~\eqref{eq:llndistance} and~\eqref{eq:llnmasse}, we finally obtain that 
 $$ \limsup_{ i \to \infty} \frac{\nu( B_{ \mathcal{X}_{i}}( \pi(Y)))}{ \mathcal{X}_{i}^{2- \varepsilon}} =0.$$
 Since $ \mathcal{X}_{i} \rightarrow 0$ a.s.~as $ i \rightarrow \infty$, this completes the proof. \endproof 

\section{Comments, extensions, conjectures and open questions}\label{sec:ext}

\subsection{Affine reinforcement}\label{sec:glpm} 

We  first investigate the extension of our results to the more general LPAM$^{\delta}$ model, 
in which vertices are chosen proportionally to an affine function of their degree. 
To describe this model, first fix a parameter  
$$\delta>-1.$$

Let $S$ be a finite tree with $n_0$ vertices. 
Define the random sequence of trees $(T^{(S),\delta}_{n})_{n \geq n_0}$ by $T^{(S),\delta}_{n_0} = S$ 
and, for $n\geq n_0$, conditionally on $T^{(S),\delta}_n$, 
the tree $T^{(S),\delta}_{n+1}$ is obtained from $T^{(S),\delta}_{n}$ 
by choosing a vertex $u \in T^{(S),\delta}_{n}$ with probability proportional to $\deg(u)+ \delta$, 
and connecting it via an edge to a new vertex. 
We call this the $LPAM^{ \delta}$ model.  It was first introduced in \cite{Mor02}. 
For $ \delta=0$ we recover  LPAM studied in the previous sections. 
The parameter $\delta$ has a dramatic impact on the geometry of $T^{(S),\delta}_{n}$ as $n \rightarrow \infty$.
For instance, it is known that in this context the maximal degree in $T^{\multimap, \delta}_{n}$ is of order $ n^{1/(2+ \delta)}$, see e.g. \cite{Mor05,Rem13}. 
Still, we conjecture that the analogs of our results hold in this setting with the appropriate modifications.
  
\begin{conjecture}[Influence of the seed]\label{thm:metric2}
For two trees $S_{1}, S_{2}$ set 
$ \displaystyle d_{\delta}(S_{1},S_{2}) = \lim_{n\to \infty} \mathrm{d_{TV}}(  {T}_{n}^{(S_{1}),\delta},  {T}_{n}^{(S_{2}),\delta})$. 
Then the function $d_{\delta}$ is a metric on 
trees with at least $3$ vertices.
\end{conjecture}

We believe that a way to prove this conjecture is to use the same observables 
(namely the number of embeddings of a certain structure in the tree at step $n$) 
as those used to prove Theorem~\ref{thm:metric}. 
However we will not pursue this goal in this paper. 

A plane version of the above algorithm may also be considered.
Assume that $T^{(S),\delta}_{n}$ is a plane tree and let us describe how to construct $T^{(S),\delta}_{n+1}$. 
Choose a vertex $u \in T^{(S),\delta}_{n}$ at random as before, and then choose uniformly at random a corner $c$ among all the corners adjacent to $u$. 
Now graft the edge leading to the new vertex of $T^{(S),\delta}_{n+1}$ in $c$. 
Using this construction, $T^{(S),\delta}_{n}$ has indeed the tree structure of the $LPAM^{ \delta}$, 
and its embedding is a uniform embedding of such a tree (assuming that this is also true  for $S$). 
Other planar versions may be considered, but we choose this one for its symmetry. 

\begin{conjecture}[Scaling limit]\label{conj}  
There exists a random compact metric space $ \mathcal{L}^{(S)}_{ \delta}$ such that the convergence
$$ {n^{ - \frac{1}{2+ \delta}}} \cdot \mathsf{Loop}(T^{(S),\delta}_{n})  \quad\mathop{\longrightarrow}^{a.s.}_{n \rightarrow \infty} \quad \mathcal{L}^{(S)}_{ \delta}$$
holds almost surely for the Gromov--Hausdorff convergence.  In addition, almost surely, the Hausdorff dimension of $ \mathcal{L}^{(S)}_{ \delta}$ is $ 2+\delta$.
\end{conjecture}

As seen previously, it is natural to scale $\mathsf{Loop}(T^{(S),\delta}_{n})$ by a factor $ n^{-1/(2+ \delta)}$, since the large degrees of $T^{(S),\delta}_{n}$ are of order $ n^{1/(2+ \delta)}$. We now give some arguments to support Conjecture~\ref{conj}. 
To simplify, as in the introduction, treat only the case $S= \multimap$. 

It may be shown that the plane $LPAM^{\delta}$ is closely related to a modification of Ford's algorithm with parameter $ \alpha=1/(2+ \delta)$. 
Ford's algorithm is a means to grow recursively a sequence of binary trees that generalizes Rémy's algorithm. 
For references see \cite{For05}. \medskip 

\textsc{Ford's algorithm: }
Fix a parameter $ \alpha \in [0,1]$. We will construct a random sequence of labeled binary trees $(\mathbf{F}_{n})_{n \geq 1}$.
Start with $\mathbf{F}_{1}$ being a binary tree with two leaves labeled $A_{0}$ and $A_{1}$. 
For $n \geq 1$, given $\mathbf{F}_{n}$, to obtain $\mathbf{F}_{n+1}$ 
we assign a weight $1- \alpha$ to each of the $n$ edges of $\mathbf{F}_{n}$ adjacent to a leaf 
and a weight $ \alpha$ to each of the $n-1$  other edges; 
then we select at random an edge $e$ proportionally to its weight and split it as in Rémy's algorithm. 
That is we place a middle vertex on $e$, to which we attach a new edge carrying a new leaf denoted by $A_{n+1}$.

\textsc{Ford's modified algorithm:} 
We consider now the following modification of Ford's algorithm, which we denote $( \tilde{ \mathbf{F}}_{n})_{n\geq 1}$. 
We proceed exactly as in Ford's algorithm except that, once the edge $e$ has been selected at step $n$, 
we first find the unique $i \in \{1,2,\ldots , n\}$ such that $e$ belongs to the geodesic joining the leaf $A_{i}$ 
to the set $ \mathsf{Span}({\tilde{F}}_{n} ; A_{0},A_{1}, \ldots, A_{i-1})$, 
then we choose a \emph{new} edge $f$ \emph{uniformly at random} on this geodesic,  
split it as in Rémy's algorithm and attach the new leaf $A_{n+1}$ to it.  
\medskip

Observe that in the case $\alpha=1/2$ both Ford's algorithm and its modified version have the same distribution as Rémy's algorithm. 
The analog of Proposition~\ref{prop:coupling} is this case is the following: 
For $ \alpha=1/(2+ \delta)$, we have the following joint equality in distribution
$$ ( \mathsf{Loop}( T^\delta_{n}) ; {n \geq 1})  \quad\mathop{=}^{(d)} \quad (\mathsf{Glu}( \tilde{\mathbf{F}}_{n}) ; {n \geq 1}).$$ 
This follows from the fact that choosing the first edge in Ford's modified algorithm amounts to choosing a vertex of $ T^\delta_{n}$ according to the LPAM$^\delta$ rule, and choosing the second edge amounts to choosing a corner of this vertex uniformly at random. We leave details to the reader. 
We also mention that the original Ford algorithm also corresponds to a plane $LPAM^{\delta}$, 
but in which corners do not play exchangeable roles
(the first corner around each vertex has weight $1-\alpha$ and all others weight $\alpha$). 
\medskip 

An analog of~\eqref{eq:cvpsremy} is known for Ford's algorithm. 
The sequence of random rescaled label trees $ n^{- \alpha}  \cdot \mathbf{F}_{n}$ 
converges almost surely towards a random compact labeled self-similar $ \R$-tree of Hausdorff dimension $1/ \alpha$ 
(belonging to the family of so-called fragmentation trees). See \cite{HM12} for details.  
A way to prove Conjecture~\ref{conj} would be to first prove analog convergences for the trees $ \tilde{ \mathbf{F}}_{n}$ arising from Ford's modified algorithm. We hope to exploit these connections in a future work.


\subsection{Connections with the Poisson boundary}

Finally, we connect the concept of the influence of the seed with the notion of the Poisson boundary of a transient Markov chain, which captures the information contained in its tail $ \sigma$-field. Consider a Markov chain $X$ on a countable state space $V$. Assume that we may write $V = V_{0} \sqcup V_{1} \sqcup V_{2} \sqcup \cdots$ in such a way that the transitions from $V_{i}$ always belong to $V_{i+1}$ for $i \geq 0$. We call this the ``layer" condition, and call $V_{i}$ a layer. In our case, $V_{n}$ is just the set of all looptrees associated with trees with $n$ vertices. In particular, this Markov chain is transient. For $x \in V_{i}$, we denote by $(X_{n}^{(x)})_{n \geq i}$ the Markov chain started from $x$. In particular, $X_{n}^{(x)} \in V_{n}$ for every $n \geq i$. For every starting points $x,y \in V$, we define the asymptotic total variation:
$$ d(x,y) := \lim_{n \to \infty} \mathrm{d_{TV}}( X_{n}^{(x)}, X_{n}^{(y)}).$$ We shall give an alternative expression for the pseudo-distance $d$ by using the Poisson boundary of the chain. The Poisson boundary of $X$ is a measurable space $ \mathcal{P} = ( E, \mathcal{A})$, which is also endowed with a family of probability measures $(\nu_{x})_{x \in V}$ such that any bounded harmonic function $h$ on $V$ can be represented as 
\begin{equation}
\label{eq:poisson}h(x) = \int \mathrm{d}\nu_{x}(\xi) \mathbf{h}(\xi)
\end{equation} where $ \mathbf{h} : E \to  \mathbb{R}$ is a bounded measurable function on $E$. The measures $ \nu_{x}$ can be interpreted as the harmonic measures on $E$ seen from $x$. 
The most classical way to construct the Poisson boundary is via the construction of the Martin boundary of the chain, we refer to \cite[Chap.~4]{Woe00} for details. We also mention that the Poisson boundary captures the information contained in the tail $ \sigma$-field of $X$. Indeed, there is a one-to-one correspondence between bounded harmonic functions $h$ and equivalence classes of bounded random variables $Z$ measurable with respect to the tail $ \sigma$-field of $X$ which is given by the formula $h(x)= \mathbb{E}_{x}[Z]$ for $x \in V$. In our setting, we have

\begin{proposition}\label{prop:dtv}For every $x,y \in V$, we have $d(x,y) = \mathrm{d_{TV}}( \nu_{x}, \nu_{y})$.
\end{proposition}

\proof  We first express $d(x,y)$  in terms of harmonic functions.  If $x,y \in V_{1} \cup V_{2} \cup \ldots \cup V_{i}$  we claim  that for $n \geq i$ we have
  \begin{eqnarray} \label{eq:harmo} & & \mathrm{d_{TV}}(X_{n}^{(x)}, X_{n}^{(y)}) =  \frac{1}{2}\sup_{h \in \mathcal{H}_{n}} |h(x)-h(y)|, \end{eqnarray}
 where $\mathcal{H}_{n} = \left\{ h : \bigsqcup_{i=1}^{n} V_{i} \to \mathbb{R}, \mbox{\ harmonic on }V_{0} \sqcup \cdots \sqcup V_{n-1} \mbox{ and } \|h\|_{\infty} \leq 1 \right\}$.
To establish this equality, remark that if we denote by $\nu_{x}^{(n)}$ the law of the first hitting point of $V_{n}$ by the chain starting from $x$ (which is also unique visited point in $V_n$ by our layer condition). Then observe that by classical potential theory, for every set $A \subset V_{n}$ we have  \begin{eqnarray*} \left|\mathbb{P}(X_{n}^{(x)} \in A) - \mathbb{P}(X_{n}^{(y)} \in A) \right| &=& \left| \int_{V_{n}} \mathrm{d}\nu^{(n)}_{x}(\xi) \mathbf{1}_{A}(\xi)-\int_{V_{n}} \mathrm{d}\nu^{(n)}_{y}(\xi) \mathbf{1}_{A}(\xi)\right|\\  &=& \frac{1}{2} \left|\int_{V_{n}} \mathrm{d}\nu^{(n)}_{x}(\xi) \left(\mathbf{1}_{A}(\xi)- \mathbf{1}_{A^c}(\xi)\right)-\int_{V_{n}} \mathrm{d}\nu^{(n)}_{y}(\xi) \left(\mathbf{1}_{A}(\xi)- \mathbf{1}_{A^c}(\xi)\right) \right| .  \end{eqnarray*}
It is plain to see that functions $h_{A} : x \mapsto \int_{V_{n}} \mathrm{d}\nu_{x}^{(n)}(\xi) (\mathbf{1}_{A}(\xi)- \mathbf{1}_{A^c}(\xi))$ are the extreme points of the convex set $ \mathcal{H}_{n}$. Hence, by convexity of $h \mapsto |h(x)-h(y)|$, we get that
$$\mathrm{d_{TV}}( X_{n}^{(x)}, X_{n}^{(y)})= \frac{1}{2} \sup_{A \subset V_{n}} \left|h_{A}(x)-h_{A}(x)\right|=\frac{1}{2}\sup_{h \in \mathcal{H}_{n}} |h(x)-h(y)|.$$
This establishes~\eqref{eq:harmo}.

By taking the limit $n \to \infty$, we get then get that \begin{equation}
\label{eq:eg}d(x,y) =  \frac{1}{2}\sup\big\{ |h(x)-h(y)|  \big\},
\end{equation}where the supremum runs over all harmonic functions $h$ on $V$ whose $\| \cdot \|_{\infty}$ norm is bounded by one. Using the Poisson representation of bounded harmonic functions~\eqref{eq:poisson}, it is a simple matter to check that the supremum on the right-hand side of~\eqref{eq:eg} is actually equal to 
 \begin{eqnarray*} \sup_{  \begin{subarray}{c}\mathbf{h}: E \to \mathbb{R} \\ \mathrm{Borel, \ } \| \mathbf{h}\|_{\infty} \leq 1 \end{subarray}} \left|\int \mathrm{d}\nu_{x}(\xi) \mathbf{h}(\xi) - \int \mathrm{d}\nu_{y}(\xi) \mathbf{h}(\xi) \right| &=& 2\sup_{ A \subset E, \mathrm{\ Borel}} \left|\int \mathrm{d}\nu_{x}(\xi)  \mathbf{1}_{A}(\xi) - \int \mathrm{d}\nu_{y}(\xi)  \mathbf{1}_{A}(\xi) \right|\\ & =& 2\sup_{A \subset E, \mathrm{\ Borel}}|\nu_{x}(A)-\nu_{y}(A)| = 2\mathrm{d_{TV}}(\nu_{x}, \nu_{y}). \end{eqnarray*}
 This completes the proof.
 \endproof 
 
In view of Proposition~\ref{prop:dtv}, a natural open question raised by our work is the following.
\begin{open} Is the measured space of scaling limits of discrete looptrees isomorphic to the Poisson boundary of the chain of planar preferential attachment trees thus implying \eqref{eq:conjintro}? Or, equivalently, are all asymptotic events of the chain measurable with respect to the scaling limit $ \mathcal{L}$?\end{open}

In particular, we believe that for every decorated tree $ \bt$, the limiting value of the martingale $M^{(S)}_{\bt}(n)$ used to prove Theorem~\ref{thm:metric} is a measurable function of $ \mathcal{L}^{(S)}$. 

\vspace{-0.1cm}
\bibliographystyle{siam}

\begin{thebibliography}{10}

\bibitem{ABBG12}
{\sc L.~Addario-Berry, N.~Broutin, and C.~Goldschmidt}, {\em The continuum
  limit of critical random graphs}, Probab. Theory Related Fields, 152 (2012),
  pp.~367--406.

\bibitem{Ald91a}
{\sc D.~Aldous}, {\em The continuum random tree. {I}}, Ann. Probab., 19 (1991),
  pp.~1--28.

\bibitem{Ald94a}
{\sc D.~Aldous}, {\em Recursive self-similarity for random trees, random
  triangulations and {B}rownian excursion.}, Ann. Probab., 22 (1994),
  pp.~527--545.

\bibitem{Ath69}
{\sc K.~B. Athreya}, {\em On a characteristic property of {P}olya's urn},
  Studia Sci. Math. Hungar., 4 (1969), pp.~31--35.

\bibitem{AB99}
{\sc A.-L. Barab{\'a}si and R.~Albert}, {\em Emergence of scaling in random
  networks}, Science, 286 (1999), pp.~509--512.

\bibitem{BRST01}
{\sc B.~Bollob{\'a}s, O.~Riordan, J.~Spencer, and G.~Tusn{\'a}dy}, {\em The
  degree sequence of a scale-free random graph process}, Random Structures
  Algorithms, 18 (2001), pp.~279--290.

\bibitem{BMR14v3}
{\sc S.~Bubeck, E.~Mossel, and M.~Z. R\'acz}, {\em On the influence of the seed
  graph in the preferential attachment model}, Preprint available on arxiv,
  http://arxiv.org/abs/1401.4849v3,  (2014).

\bibitem{BMR14v2}
\leavevmode\vrule height 2pt depth -1.6pt width 23pt, {\em On the influence of
  the seed graph in the preferential attachment model}, Preprint available on
  arxiv, http://arxiv.org/abs/1401.4849v2,  (2014).

\bibitem{BBI01}
{\sc D.~Burago, Y.~Burago, and S.~Ivanov}, {\em A course in metric geometry},
  vol.~33 of Graduate Studies in Mathematics, American Mathematical Society,
  Providence, RI, 2001.

\bibitem{CMP14}
{\sc B.~Chauvin, C.~Mailler, and N.~Pouyanne}, {\em Smoothing equations for
  large {P}{\'o}lya urns.}, Journal of Theoretical Probability,  (To appear).

\bibitem{CH13}
{\sc N.~Curien and B.~Haas}, {\em The stable trees are nested}, Probab. Theory
  Related Fields, 157 (2013), pp.~847--883.

\bibitem{CK13+}
{\sc N.~Curien and I.~Kortchemski}, {\em Percolation on random triangulations
  and stable looptrees}, Preprint available on arxiv,
  http://arxiv.org/abs/1307.6818,  (Submitted).

\bibitem{CK13}
\leavevmode\vrule height 2pt depth -1.6pt width 23pt, {\em Random stable
  looptrees}, arXiv:1304.1044,  (submitted).

\bibitem{Eva08}
{\sc S.~N. Evans}, {\em Probability and real trees}, vol.~1920 of Lecture Notes
  in Mathematics, Springer, Berlin, 2008.
\newblock Lectures from the 35th Summer School on Probability Theory held in
  Saint-Flour, July 6--23, 2005.

\bibitem{For05}
{\sc D.~J. Ford}, {\em Probabilities on cladograms: Introduction to the alpha
  model}, Preprint. Available at arXiv:math/0511246v1.

\bibitem{HM12}
{\sc B.~Haas and G.~Miermont}, {\em Scaling limits of {M}arkov branching trees,
  with applications to {G}alton-{W}atson and random unordered trees}, Ann. of
  Probab., 40 (2012), pp.~2589--2666.

\bibitem{LG05}
{\sc J.-F. Le~Gall}, {\em Random trees and applications}, Probability Surveys,
  (2005).

\bibitem{LGICM}
\leavevmode\vrule height 2pt depth -1.6pt width 23pt, {\em Random geometry on
  the sphere}, To appear in the Proceedings of ICM 2014, Seoul, available on
  arXiv,  (2014).

\bibitem{Mat95}
{\sc P.~Mattila}, {\em Geometry of sets and measures in {E}uclidean spaces},
  vol.~44 of Cambridge Studies in Advanced Mathematics, Cambridge University
  Press, Cambridge, 1995.
\newblock Fractals and rectifiability.

\bibitem{Mie09}
{\sc G.~Miermont}, {\em Tessellations of random maps of arbitrary genus}, Ann.
  Sci. \'Ec. Norm. Sup\'er. (4), 42 (2009), pp.~725--781.

\bibitem{Mor02}
{\sc T.~F. M{\'o}ri}, {\em On random trees}, Studia Sci. Math. Hungar., 39
  (2002), pp.~143--155.

\bibitem{Mor05}
{\sc T.~F. M{\'o}ri}, {\em The maximum degree of the {B}arab\'asi-{A}lbert
  random tree}, Combin. Probab. Comput., 14 (2005), pp.~339--348.

\bibitem{PRR14}
{\sc E.~A. Pek{\"o}z, A.~R{\"o}llin, and N.~Ross}, {\em Joint degree
  distributions of preferential attachment random graphs}, Preprint available
  on arxiv, http://arxiv.org/abs/1402.4686,  (2014).

\bibitem{Pit06}
{\sc J.~Pitman}, {\em Combinatorial stochastic processes}, vol.~1875 of Lecture
  Notes in Mathematics, Springer-Verlag, Berlin, 2006.
\newblock Lectures from the 32nd Summer School on Probability Theory held in
  Saint-Flour, July 7--24, 2002, With a foreword by Jean Picard.

\bibitem{Rem85}
{\sc J.-L. R{\'e}my}, {\em Un proc\'ed\'e it\'eratif de d\'enombrement d'arbres
  binaires et son application \`a leur g\'en\'eration al\'eatoire}, RAIRO
  Inform. Th\'eor., 19 (1985), pp.~179--195.

\bibitem{Szy87}
{\sc J.~Szyma{\'n}ski}, {\em On a nonuniform random recursive tree}, in Random
  graphs '85 ({P}ozna\'n, 1985), vol.~144 of North-Holland Math. Stud.,
  North-Holland, Amsterdam, 1987, pp.~297--306.

\bibitem{Rem13}
{\sc R.~van~der Hofstad}, {\em Lecture notes random graphs and complex
  networks}, In preparation, R. van der Hofstad (May 2013). Random graphs and
  complex networks. http://www.win.tue.nl/~rhofstad/NotesRGCN.pdf,  (2013).

\bibitem{Woe00}
{\sc W.~Woess}, {\em Random walks on infinite graphs and groups}, vol.~138 of
  Cambridge Tracts in Mathematics, Cambridge University Press, Cambridge, 2000.

\end{thebibliography}

\end{document}